\documentclass[11pt]{amsart}
\usepackage{url, 
  amssymb,setspace, mathrsfs,fontenc}
  \usepackage[alphabetic]{amsrefs}
\usepackage{fullpage} 
\usepackage{color}
\usepackage{tikz}
\usepackage[all]{xy}

\usetikzlibrary{decorations.markings,snakes}
\newtheorem{thm}{Theorem}
\newtheorem{lem}[thm]{Lemma}
\newtheorem{prop}[thm]{Proposition}

\newtheorem{cor}[thm]{Corollary}
\newtheorem{defe}[thm]{Definition}
 
\theoremstyle{remark}
\newtheorem{rem}[thm]{Remark}
\newtheorem{exam}[thm]{Example}
\newtheorem{ass}[thm]{Assumption}
\newtheorem{ntn}[thm]{Notation}

\DefineSimpleKey{bib}{myurl}
\newcommand\myurl[1]{\url{#1}}
\BibSpec{webpage}{
  +{}{\PrintAuthors} {author}
  +{,}{ \textit} {title}
  +{}{ \parenthesize} {date}
  +{,}{ \myurl} {myurl}
}

\newcommand{\nc}{\newcommand}

\nc{\ssec}{\subsection}

\nc{\on}{\operatorname}

\nc {\cA}{\mathcal{A}}
\nc {\cG} {\mathcal{G}}
\nc {\cK}{\mathcal{K}}
\nc {\cC} {\mathcal{C}}
\nc {\cL} {\mathcal{L}}
\nc {\cE} {\mathcal{E}}
\nc {\cM} {\mathcal{M}}
\nc {\cO}{\mathcal{O}}
\nc {\cF}{\mathcal{F}}
\nc {\cZ}{\mathcal{Z}}

\nc {\bZ}{\mathbb{Z}}

\nc {\uG} {\underline{G}}
\nc {\cB}{\mathcal{B}}
\nc{\rat}{\mathrm{rat}}
      
\nc {\fk}{\mathfrak{k}}
\nc {\fI}{\mathfrak{i}}
\nc {\fg} {\mathfrak{g}}
\nc {\fl} {\mathfrak{l}}
\nc {\fn} {\mathfrak{n}}
\nc {\cP} {\mathcal{P}}
\nc {\fz} {\mathfrak{z}}
\nc {\fc}{\mathfrak{c}}
\nc {\fd}{\mathfrak{d}}
\nc {\fh}{\mathfrak{h}}
\nc {\fp}{\mathfrak{p}}
\nc {\ft}{\mathfrak{t}}

\nc{\tg} {\mathtt{g}}
\nc {\tP}{\mathcal{P}}
\nc {\hfg} {\widehat{\fg}}
\nc {\hG} {\check{G}}

\nc {\bGm} {\mathbb{G}_m}
\nc{\bC}{\mathbb{C}}
\nc{\bV}{\mathbb{V}}
\nc{\bP}{\mathbb{P}}
\nc{\bA}{\mathbb{A}}
\nc {\bQ}{\mathbb{Q}}
\nc {\bR}{\mathbb{R}}
\nc{\Sl}{\mathfrak{sl}}
\nc{\Gl}{\mathfrak{gl}}
\nc {\So}{\mathfrak{so}}
\nc{\ra}{\rightarrow}

\nc {\tU}{\tilde{U}}
\nc {\tSym}{\widetilde{Sym}}

\nc {\Bun}{\mathrm{Bun}}
\nc {\Fun}{\mathrm{Fun}}
\nc {\crit}{\mathrm{crit}}
\nc {\Ind}{\mathrm{Ind}}
\nc {\Vac}{\mathrm{Vac}}
\nc {\gr}{\mathrm{gr}}
\nc {\ad}{\mathrm{ad}}
\nc {\Sym}{\mathrm{Sym}}
\nc {\Ram}{\mathrm{Ram}}
\nc {\Op}{\mathrm{Op}}
\nc {\Hitch}{\mathrm{Hitch}}
\nc {\fb}{\mathfrak{b}}
\nc{\cDt}{\mathcal{D}^\times}
\nc{\tp}{\mathtt{p}}
\nc {\tk}{\mathtt{k}}
\nc{\cN}{\mathcal{N}}
\nc {\cY}{\mathcal{Y}}
\nc {\loc}{\mathrm{loc}}
\nc {\Res}{\mathrm{Res}}
\nc {\Hom}{\mathrm{Hom}}
\nc {\End}{\mathrm{End}}
\nc {\Perm}{\mathrm{Perm}}
\nc{\Dec}{\mathrm{Dec}}
\nc {\reg}{\mathrm{reg}}
\nc {\GL}{\mathrm{GL}}
\nc{\SL}{\mathrm{SL}}
\nc{\fm}{\mathfrak{m}}
\nc {\fa}{\mathfrak{a}}
\nc {\Spec}{\mathrm{Spec}}
\nc {\cD}{\mathcal{D}}
\nc {\nilp}{\mathrm{nilp}}
\nc {\tN}{\widetilde{\cN}}
\nc {\tO}{\widetilde{\cO}}
\nc {\Sp}{\mathrm{Sp}}
\nc{\cV}{\mathcal{V}}
\nc {\RS}{\mathrm{RS}}
\nc {\ord}{\mathrm{ord}}
\nc {\fu}{\mathfrak{u}}
\nc {\sign}{\mathrm{sign}}

\begin{document} 
\title{On the image of the parabolic Hitchin map} 

\begin{abstract} 
We determine the image of the  (strongly) parabolic Hitchin map for all parabolics in classical groups and $G_2$. Surprisingly, we find that the image is isomorphic to an affine space in all cases, except for certain ``bad parabolics'' in type $D$, where the image can be singular. \end{abstract} 

\author{David Baraglia} 
\author{Masoud Kamgarpour}


\address{Department of Mathematics, University of Adelaide} 
\email{david.baraglia@adelaide.edu.au}

\address{School of Mathematics and Physics, The University of Queensland} 
\email{masoud@uq.edu.au}

\maketitle

\tableofcontents

\section{Introduction} \label{s:intro}

The Hitchin map on the moduli space of Higgs bundles \cite{Hitchin1, Hitchin} lies at the crossroad of geometry, Lie theory, and mathematical physics. It has found remarkable applications in non-abelian Hodge theory \cite{Simpson90, Simpson} and the Langlands program \cite{BD, Ngo}. The original construction of Hitchin was on compact Riemann surfaces. The generalisation to non-compact surfaces has been developed in, e.g., \cite{SimpsonNonCompact, Boalch, BKV}. The relevant objects in the non-compact case are \emph{parabolic} Higgs bundles.\footnote{To be more precise, one should consider \emph{parahoric Higgs bundles}; cf. \cite{BKV}. In this text, however, we restrict ourselves to the parabolic case.}
The aim of this paper is to determine the image of the Hitchin map  on the moduli space of parabolic Higgs bundles.  We start by stating a ``local version'' of our result, since this can be formulated in standard Lie-theoretic language. 
 
\subsection{Main local result}\label{s:mainlocal}
Let $G$ be a simple simply connected algebraic group over $\bC$ and let $\fg$ denote its Lie algebra. 
According to a theorem of Chevalley, the invariant ring $\bC[\fg]^G$ is isomorphic to a graded polynomial ring with homogenous generators $Q_1,\cdots, Q_n$. These generators are not canonical; however, their degrees $d_1, \cdots , d_n$ are, up to re-ordering, canonical invariants of $\fg$ called the \emph{fundamental degrees}.

Let $\cO=\bC[\![t]\!]$ denote the local ring of formal power series over $\bC$ in the variable $t$ and let $\cK=\bC(\!(t)\!)$ denote the quotient field of $\cO$. Consider the map 
\[
\chi: \fg(\cK)\ra \bA_\cK^n, \quad \quad x\mapsto \chi(x)=(Q_1(x),\cdots, Q_n(x)).
\]
Here, $\bA_\cK^n\simeq \cK^n$ is the $n$-dimensional affine space over $\cK$. 
We are interested in determining the image of certain subspaces of $\fg(\cK)$ under this map. These subspaces are closely related to parabolic subalgebras of $\fg$. Let $\fp\subseteq \fg$ be a parabolic subalgebra, and let $\fp=\fl\oplus \fn$ denote its Levi decomposition. 
Let
\[
\tp=\fp\oplus t\fg(\cO)\quad \textrm{and}\quad \tp^\perp=t^{-1}\fn\oplus \fg(\cO).
\]
One of our main goals is to describe $\chi(\tp^\perp)$ explicitly.

\begin{rem} The map $\chi$ is a local analogue of the Hitchin map on the moduli space of Higgs bundles. As we explain  below, the space $\tp^\perp$ may be interpreted as the space of germs of (strongly) parabolic Higgs fields around a marked point with parabolic structure. Thus, $\chi( \tp^\perp)$ is a local analogue of the image of the parabolic Hitchin map. This is our motivation for studying $\chi( \tp^\perp)$. Our notation is motivated by the fact that $\tp$ is a parahoric subalgebra of $\fg(\cK)$ and $\tp^\perp$ is its annihilator under the canonical non-degenerate pairing on $\fg(\cK)$; cf. \cite{BKV}.
\end{rem} 

\begin{rem}\label{r:topology}
 To make the above problem more tractable, we consider the closure of $\chi(\tp^\perp)$ in $\bA_\cK^n$.  Let us explain what we mean by the word ``closure'' here. In fact, it is easy to see that $\chi(\tp^\perp)$ takes values in $t^{-m}\bA^n_\cO$ for some positive integer $m$. 
Now we think of $t^{-m}\bA^n_\cO$  as a pro-algebraic variety over $\bC$
  \[
t^{-m} \bA^n_\cO=t^{-m}\varprojlim_{k\geq 0} \bA^n_{\cO/t^k},
 \]
 where $\cO/t^k=\bC[t]/t^k$. By definition, a subset $U\subseteq t^{-m}\bA^n_\cO$ is dense if and only if $U/t^k\subseteq t^{-m}\bA^n_{\cO/t^k}$
  is dense in the Zariski topology for all $k\geq 0$. 
   \end{rem}

\begin{ntn} Let $\fa_{\fg,\fp}$  denote the closure of $\chi(\tp^\perp)$ in $\bA^n_{\cK}$.\footnote{It may be that $\chi(\tp^\perp)$ is closed in $\bA^n_\cK$, but we do not know how to prove this, except in Type $A$.} 
\end{ntn} 

 To give a description of $\fa_{\fg,\fp}$, we will fix a specific basis of invariant polynomials in the following way: 

\begin{ass}\label{ass:main}  
If $\phi$ is an $n\times n$ matrix, we write
\[
\det(\lambda I-\phi)=\lambda^n+c_1(\phi)\lambda^{n-1}+\cdots + c_n(\phi).
\]
We use the following generators for the invariant ring: 
\begin{itemize} 
\item Type $A_{n-1}$: identify $\fg$ with traceless $n \times n$ matrices and use  $(c_2, \cdots, c_n)$ as generators. 
\item Type $B_n$: identify $\fg$ with skew-symmetric $(2n+1)\times (2n+1)$ matrices and use $(c_2, c_4, \cdots, c_{2n})$.
\item Type $C_n$: identify $\fg$ with $2n\times 2n$ matrices preserving a symplectic form and use $(c_2, c_4,\cdots, c_{2n})$. 
\item Type $D_n$: identify $\fg$ with skew-symmetric $2n\times 2n$ matrices and use $(c_2, c_4, \cdots, c_{2n-2}, p_n)$. Here $p_n$ denotes the Pfaffian, which satisfies $p_n^2 = c_{2n}$. 
\item Type $G_2$: take the embedding $G_2\subset SL_7$ and use $(c_2,c_6)$; cf. \cite[\S 5.2]{HitchinG2}. 
\end{itemize} 
\end{ass}

\begin{rem}\label{rem:ordering}
Note that in all cases except type $D$, our chosen invariant polynomials are ordered by degree. However, in type $D$, {\em we do not order our invariant polynomials by degree}. This convention is required for the statement of Theorem \ref{t:main}.
\end{rem}

It turns out that $\fa_{\fg,\fp}$ admits a simple description in most cases under consideration. However, for certain parabolics subgroups in type $D$, the structure of $\fa_{\fg,\fp}$ is considerably more complicated. In the introduction, we shall only give the description of $\fa_{\fg,\fp}$ in the ``good" cases; i.e., those which satisfy Assumption \ref{ass:main2} below. The precise description of $\fa_{\fg,\fp}$ in the ``bad" cases is left to \S \ref{s:TypeD}.

\begin{ass} \label{ass:main2} 
Assume that the pair $(\fg , \fp)$ is one of the following types:
\begin{enumerate} 
\item $\fg$ is of type $A$, $B$, $C$, or $G_2$ and $\fp$ is arbitrary, or,
\item $\fg$ is of type $D_{n}$ and $\fp$ is the parabolic subalgebra of $\mathfrak{so}_{2n}$ corresponding to a subset $S = \{ \alpha_{i_1} , \alpha_{i_2} , \dots , \alpha_{i_k} \}$ of simple roots, with $i_1 < i_2 < \cdots < i_k$ satisfying $2(n - i_k) \ge \max ( i_1 , i_2 - i_1 , \dots , i_k - i_{k-1}, 4)$, where the simple roots $\alpha_1 , \dots , \alpha_n$ of $D_{n}$ are as indicated in the following diagram:

\begin{center}
\begin{tikzpicture}
\draw [thick] (2,0) -- (3.4,0) ;
\draw [thick] (4.6,0) -- (6,0) ;
\draw [thick] (6,0) -- (7,0.5);
\draw [dotted, thick] (3.6,0) -- (4.4,0);
\draw [fill] (2,0) circle(0.1);
\draw [fill] (3,0) circle(0.1);
\draw [fill] (5,0) circle(0.1);
\draw [fill] (6,0) circle(0.1);
\draw [fill] (7,0.5) circle(0.1);
\draw [fill] (7,-0.5) circle(0.1);

\node at (2,0.4) {$\alpha_1$};
\node at (3,0.4) {$\alpha_2$};
\node at (6,0.4) {$\alpha_{n-2}$};
\node at (7.7,0.5) {$\alpha_{n-1}$};
\node at (7.7,-0.5) {$\alpha_{n}$};

\draw [thick] (6,0) -- (7,-0.5) ;

\end{tikzpicture}
\end{center}
To understand this condition better, set $r_j = i_j - i_{j-1}$ and $s = n-i_k$. Then the Levi $\fl$ of $\fp$ has the form:
\[
\fl \cong \mathfrak{gl}_{r_1} \times \cdots \times \mathfrak{gl}_{r_k} \times \mathfrak{so}_{2s},
\]
and our assumption says $2s \ge \max(r_1 , r_2 , \dots , r_k , 4)$. 
\end{enumerate} 
\end{ass} 

Let $m_1\leq \cdots\leq  m_n$ denote the fundamental degrees of the Levi subalgebra $\fl\subseteq \fp$. Note that as $\fl$ is not semisimple, it will have $1$ as a fundamental invariant (possibly with multiplicity). Now we have:

\begin{thm} \label{t:main} Under Assumptions \ref{ass:main} and \ref{ass:main2}, we have 
\[
\fa_{\fg,\fp}= \bigoplus_{i=1}^n  t^{-d_i+m_i}\cO.
\]
\end{thm} 

\begin{rem}
In this theorem, we always put the fundamental invariants of $\fl$ in increasing order $m_1 \leq \cdots \leq m_n$, but the fundamental invariants of $\fg$ are not necessarily in increasing order if $\fg$ is of type $D$, see Remark \ref{rem:ordering}.
\end{rem}

This theorem can be considered as a refinement of a part of \cite[prop. 3.10]{Zhu}.\footnote{A quantum analogue of Zhu's result appeared in \cite{CK}.} We prove the inclusion $\subseteq$ of Theorem \ref{t:main}  by a direct computation. These computations are similar to the ones in \cite[\S 9]{KL}; see also \cite{Rohith} for an alternative approach in Type $C$. In Type $A$, we prove the inclusion $\supseteq$ of Theorem \ref{t:main} by constructing an appropriate companion matrix. We do not know how to construct these companion matrices in other types. Thus, to establish the inclusion  $\supseteq$ beyond Type $A$, we have to resort to global methods (see below).

\begin{exam}
Suppose $\fg=\Sl_4$ and let 
\[
\fp=
\left(
\begin{matrix} 
* & * & * & * \\
* & * & * & * \\
* & * & *& * \\
0 & 0 & 0 & *
\end{matrix} 
\right) 
= 
\fl\oplus \fn =
\left(
\begin{matrix} 
* & * & * & 0 \\
* & * & * & 0 \\
* & * & * & 0 \\
0 & 0 & 0 & *\\
\end{matrix} 
\right) \oplus 
\left(
\begin{matrix} 
0 & 0 & 0 & * \\
0 & 0 & 0 & * \\
0 & 0 & 0 & * \\
0 & 0 & 0& 0 \\
\end{matrix} 
\right) \implies
\tp^\perp =
\left(
\begin{matrix} 
\cO & \cO & \cO & t^{-1}\cO \\
\cO & \cO & \cO & t^{-1}\cO \\
\cO & \cO & \cO & t^{-1}\cO \\
\cO & \cO & \cO & \cO \\
\end{matrix} 
\right)
\] 
Here, we have $d_1=2$, $d_2=3$, $d_3=4$ and $m_1=1$, $m_2=2$, $m_3=3$. 
An explicit computation shows that $\fa_{\fg,\fp} = t^{-1}\cO \oplus t^{-1}\cO\oplus t^{-1}\cO$ in agreement with the above theorem. 
\end{exam}

\begin{rem} 
Let us explain the necessity of using the coefficients of characteristic polynomials. Indeed, in the example above, it is easy to see that there exists $A\in \tp^\perp$ such that $\on{Tr}(A^4)$ has a pole of order  $2$. Thus, Theorem \ref{t:main} does not hold if we use traces of powers of matrices as algebraically independent generators for the invariant ring of $\Sl_4$. 
\end{rem} 

\begin{rem} 
Consider a parabolic in Type $D$ that does not satisfy Assumption \ref{ass:main2}. As shown in \S \ref{s:TypeD}, in this case, $\fa_{\fg,\fp}$ has a complicated description, regardless of what basis of invariant polynomials we choose. In fact, we show that the global analogue of $\fa_{\fg,\fp}$ (denoted by $\cA_{G,P}$ below) can be \emph{singular} for these parabolics. This is in sharp contrast with the fact that when the assumptions are satisfied, $\cA_{G,P}$ is isomorphic to an affine space (Theorem \ref{t:main2}). 
\end{rem} 

\begin{rem} 
In \cite{GW}, Gukov and Witten call the image of the parabolic Hitchin map, the \emph{fingerprint of surface operators} and identify it as one of the key ingredients for checking Langlands duality of Hitchin fibres. In \S 4 of \emph{op. cit.}, they  conjecture that it should be possible to describe the image of the parabolic Hitchin map using the Kazhdan-Lusztig map from nilpotent orbits to conjugacy classes in the Weyl group \cite[\S 9]{KL}. Unfortunately, we were not able to make this conjecture precise. 
\end{rem}

\subsection{Main global result} 
Let $X$ be a smooth projective curve over $\bC$ of genus $g>1$ and let $\Omega$ denote the canonical bundle on $X$. Let $P_i$, $i=1,\cdots, k$, be parabolic subgroups of $G$. Let $x_i$, $i=1,\cdots, k$, be marked points on $X$. A parabolic $G$-bundle of type $(P_i,x_i)_{i=1,\cdots, k}$ is a $G$-bundle $\cE$ on $X$ equipped with  $P_i$-reduction $\cE_{x_i}$ of $\cE$ at $x_i$; cf. \cite{LS}. Here, $\cE_{x_i}$ denotes the restriction of $\cE$ to $x_i$. 

\begin{ntn} In what follows, for the ease of notation, we assume that we have only one marked point $x\in X$ and one parabolic $P$. The generalisation to finitely many parabolics introduces no additional complications. We let $\Bun_{G,P}$ denote the moduli stack of parabolic $G$-bundles on $X$ of type $(P,x)$. 
\end{ntn} 

A (strongly) parabolic $G$-\emph{Higgs bundle} of type $(P,x)$ is a parabolic $G$-bundle of type $(P,x)$ together with a section $\phi \in \Gamma(X, \ad^*_\cE\otimes \Omega(x))$ whose residue at $x$ lies in the nilpotent radical $\fn$. We let $\cM_{G,P}$ denote the moduli stack of parabolic $G$-Higgs bundles on $X$ of type $(P,x)$. Observe that, with respect to a local trivialisation of $E$ and of $\Omega$ around $x$, the germ of $\phi$ at $x$ is given by an element of $\tp^\perp = t^{-1} \fn \oplus \fg(\cO)$. It is in this sense that $\tp^\perp$ corresponds to local parabolic Higgs fields.

\begin{exam} 
  If $G=\GL_n$, then a $G$-bundle is the same as a rank $n$ vector bundle and a parabolic $G$-bundle (of type $(P,x)$) is a pair $(\cE,\cF_*)$ consisting of a vector bundle $\cE$ of rank $n$ and a flag $\cF_*$ of subspaces of $\cE_x$
\[
\{ 0 \} = \cF_0 \subset \cF_1 \subset \cdots \subset \cF_k =\cE_x
\]
such that the stabiliser of this flag is $P$. A parabolic $G$-Higgs bundle of type $(P,x)$ consists of a vector bundle $\cE$ and flag $\cF_*$ as above, together with a section $\phi \in \Gamma(X , \End(\cE) \otimes \Omega(x))$ such that the residue $\Res_x(\phi) \in \End(\cE)_x$ of $\phi$ at $x$ maps $\cF_i$ to $\cF_{i-1}$.
\end{exam} 

\begin{defe} 
The Hitchin map $h_{G,P}$ on $\cM_{G,P}$ is defined by 
\[
h_{G,P}(\cE,\phi)=(Q_1(\phi), \cdots Q_n(\phi)) \in \bigoplus_{i=1}^n \Gamma(X, \Omega^{d_i}(d_i.x)).
\]
\end{defe} 

\begin{ntn}
Let $\displaystyle \cA_{G,P} \subseteq \bigoplus_{i=1}^n \Gamma(X, \Omega^{d_i}(d_i.x))$ denote the Zariski closure of the image of $h_{G,P}$.
\end{ntn}

 The finite dimensional affine variety $\cA_{G,P}$ can be considered as the global analogue of $\fa_{\fg,\fp}$, introduced in \S \ref{s:mainlocal}. As in the local setting, we prove that $\cA_{G,P}$ admits a simple description in most cases under consideration: 

\begin{thm}\label{t:main2}
 Under Assumptions \ref{ass:main} and \ref{ass:main2}, we have 
\[
\cA_{G,P} = \bigoplus_i \Gamma(X, \Omega^{d_i}((d_i-m_i).x)). 
\]
\end{thm} 

We refer the reader to \S \ref{ss:typeDglobal} for the description of $\cA_{G,P}$ when $P$ is a ``bad parabolic" in Type $D$. Our approach to proving Theorem \ref{t:main2} is as follows: the inclusion $\subseteq$ is an immediate consequence of the corresponding inclusion $\subseteq$ in Theorem \ref{t:main}. To prove the opposite inclusion, we use the formula 
 \begin{equation} \label{eq:dim}
 \dim(\cA_{G,P})=\dim(\Bun_{G,P})=\frac{1}{2} \dim(\cM_{G,P})=(g-1)\dim(G)+\dim(G/P).
 \end{equation}
 established in \cite{BKV} (for arbitrary $G$ and $P$). Comparing Equation \eqref{eq:dim} with the dimension of $\Gamma(X, \Omega^{d_i}((d_i-m_i).x))$ establishes the theorem. In \S \ref{s:completion}, we use Theorem \ref{t:main2} to finish the proof of Theorem \ref{t:main}. For bad parabolics in type $D$, a similar kind of strategy is used to solve the local and global problems.

\begin{rem}
In the literature, one also finds the notion of \emph{weakly} parabolic Higgs bundles, where we only require that the residue of $\phi$ is in $\fp$ \cite{LogaresMarten}. The image of the Hitchin map for weakly parabolic Higgs bundles is easy to compute: it equals 
 $\bigoplus_{i=1}^n \Gamma(X, \Omega^{d_i}(d_i.x))$, regardless of the parabolic.
\end{rem}

\begin{exam} 
If $P$ is a Borel subalgebra, then $m_i=1$ for all $1\leq i \leq n$. In this case, the above theorem is well-known and easy to establish. In the case of type A and arbitrary parabolic, the above theorem is proved in \cite{SS}. However, it is unclear whether the methods of \cite{SS} can be extended to other groups. 
\end{exam}

 \begin{rem} 
We note that the companion matrix construction used in the proof of Theorem \ref{t:main} for Type $A$ (see \S \ref{ss:section}) can be globalised to define a parabolic analogue of the Kostant-Hitchin section and prove Theorem \ref{t:main2} in Type $A$ without resorting to  Formula \eqref{eq:dim}. It is an interesting open problem to generalise this section to other types. 
\end{rem} 

\subsection{Acknowledgements} We thank Dima Arinkin, David Ben-Zvi, Tsao-Hsien Chen, Sergei Gukov, Arun Ram, Rohith Varma, Zhiwei Yun and Xinwen Zhu for helpful discussions. The authors were supported by Australian Research Council DECRA Fellowships.


\section{Proof of Theorem \ref{t:main} for Type A}\label{s:TypeA}

\subsection{Setup}
Let $V = \mathbb{C}^n$ be the standard representation of $\fg = \Gl_n$ and let $\fp \subseteq \fg$ be a parabolic subalgebra. Then $\fp$ can be described as the subalgebra of $\fg \cong \on{End}(V)$ preserving a flag of subspaces
\begin{equation*}
\{ 0 \} = F_0 \subset F_1 \subset F_2 \subset \dots \subset F_k = V.
\end{equation*}
In other words, an endomorphism $T : V \to V$ is in $\fp$ if and only if $T(F_j) \subseteq F_j$ for all $j$. The type of parabolic $\fp$ is determined by the dimensions of subspaces in the flag. More specifically, let $n_j = \dim(F_j/F_{j-1})$ so that $n = n_1 + n_2 + \dots + n_k$ is a partition of $n$. Choose splittings $F_j = F_{j-1} \oplus V_j$ of the flag so that
\begin{equation*}
V = V_1 \oplus V_2 \oplus \dots \oplus V_k.
\end{equation*}
Then $\fp, \fn$ are given by
\begin{equation*}
\begin{aligned}
\fp &= \bigoplus_{i \ge j} \Hom( V_i , V_j ), \\ 
\fn &= \bigoplus_{i > j} \Hom( V_i , V_j ).
\end{aligned}
\end{equation*}

For $\phi \in \fg(\cK)$, write the characteristic polynomial as
\begin{equation*}
\det(\lambda I-\phi)= \lambda^n+c_1(\phi)\lambda^{n-1}+\cdots + c_n(\phi) \in \cK[\lambda].
\end{equation*}
The $\lambda^{n-j}$ coefficient $c_j : \fg(\cK) \to \cK$ is a degree $j$ invariant polynomial. The $\Gl_n$ local Hitchin map is given by
\[
\chi: \fg(\cK)\ra \cK^n,\quad  \phi \mapsto (c_1(\phi),\cdots c_n(\phi)). 
\] 
Let $\fl$ be the Levi subalgebra of $\fp$, which is reductive of rank $n$. We have that
\begin{equation*}
\fl = \bigoplus_{i} \End(V_i).
\end{equation*}
Let $m_0\leq m_1\leq \cdots \leq m_{n-1}$ denote the fundamental degrees of $\fl$ and let $\delta_1 \ge \delta_2 \ge \dots \ge \delta_m$ be the conjugate partition of $(n_1 , n_2 , \dots , n_k)$. Since $\End(V_i)$ has fundamental degrees $1,2,\dots , n_i$ it is easy to see that $m_0, m_1 , \dots , m_{n-1}$ are given by the following sequence:
\begin{equation}\label{eq:msequence}
\underbrace{1,1, \dots , 1}_{\delta_1 \text{ times}}, \underbrace{2,2, \dots , 2}_{\delta_2 \text{ times}}, \dots , \underbrace{m,m, \dots , m}_{\delta_{m} \text{ times}}.
\end{equation}
Note that in the $\Sl_n$ case, the corresponding Levi is the trace-free part of $\fl$, which has rank $n-1$ and the fundamental degrees are then $m_1 \le m_2 \le \cdots \le m_{n-1}$.

Recall that $\tp =\fp+t\fg(\cO)$ and $\tp^\perp=t^{-1}\fn +\fg(\cO)$. We will show that
\begin{equation*}
\displaystyle \chi(\tp^\perp) = \bigoplus_{i=0}^{n-1} t^{-i-1+m_i}\cO.
\end{equation*}
From this we easily obtain the result in the case of $\Sl_n$. As $t\tp^\perp=\fn +t\fg(\cO)$, this is clearly equivalent to showing:
\begin{equation}\label{eq:localgln}
\displaystyle \chi(t\tp^\perp) = \bigoplus_{i=0}^{n-1} t^{m_i}\cO.
\end{equation}
We will prove Equation (\ref{eq:localgln}) by showing inclusions in both directions.

\subsection{Companion matrices and the inclusion $\bigoplus_{i=0}^{n-1} t^{m_i}\cO \subseteq \chi(t\tp^\perp)$.}\label{ss:section} 
Recall that $c_j : \fg(\cK) \to \cK$ is an invariant polynomial of degree $j$. Let $c_j(A_1 , A_2 , \dots , A_j)$ denote the corresponding degree $j$ invariant symmetric multilinear map, defined so that $c_j(A) = c_j(A,A, \dots , A)$. It is easy to see that $c_j$ is the trace of the linear map $\wedge^j V \to \wedge^j V$ given by
\begin{equation}\label{eq:trace}
v_1 \wedge \dots \wedge v_{j} \mapsto \frac{(-1)^j}{j!} \sum_{\sigma} A_{\sigma(1)}(v_1) \wedge \dots \wedge A_{\sigma(j+1)}(v_{j+1}),
\end{equation}
where the sum is over all permutations of $1, \dots , j$. Indeed, the above expression is clearly symmetric and if we set $A_1 = \dots = A_{j} = A$, we recover the definition of $c_j(A)$.

Choose basis $e_{1,1} , \dots , e_{1,n_1}$ for $V_1$ and so on up to $e_{k,1}, \dots , e_{k,n_k}$ for $V_k$. Let $e_{i,i'}^*$ be the dual basis vectors. Then a basis for $\fp$ is given by $\{ e_{i,i'}^* \otimes e_{j,j'} \}_{i \ge j}$ and a basis for $\fn$ by $\{ e_{i,i'}^* \otimes e_{j,j'} \}_{i > j}$. 

\begin{lem}\label{l:cycles}
Given two sequences $(i_1,i'_1) , \dots , (i_{j} , i'_{j})$ and $(j_1,j'_1) , \dots , (j_{j} , j'_{j})$ of length $j$, let $x_u = e_{i_u , i'_u}^* \otimes e_{j_u , j'_u}$. Then
\begin{equation*}
c_j( x_1 , \dots , x_{j}) = \begin{cases} (-1)^j \sign(\sigma) & \text{if } (i_1,i'_1) , \dots , (i_{j} , i'_{j}) \text{ are distinct and } \\ & (j_u,j'_u) = (i_{\sigma(u)} , i'_{\sigma(u)}) \text{ for a permutation } \sigma, \\ 0 & otherwise.  \end{cases}
\end{equation*}
\end{lem}
\begin{proof}
Straightforward, using that $c_j$ is the trace of the map given by (\ref{eq:trace}).
\end{proof}

Let $S(n_1, \dots , n_k) = \{ (a,b) \in \mathbb{Z}^2 \; | \; 1 \le a \le k, \; 1 \le b \le n_a \}$. For each $(a,b) \in S$ we have a corresponding basis vector $e_{(a,b)} \in V_a$. Define a function $s : \{ 0, 1 ,  \dots , n-1 \} \to S$ as follows. Let $s(0) = (k,1)$. Then we obtain $s(i+1)$ from $s(i)$ in the following manner. Suppose that $s(i) = (\alpha(i) , \beta(i) )$. If there is an element of $S$ of the form $(\alpha' , \beta(i) )$ for $\alpha' < \alpha(i)$ then we let $s(i+1)$ be the element of this form which has maximal $\alpha'$. If there is no such element then we look for elements of $S$ of the form $(\alpha' , \beta(i)+1)$. If there is such an element we take $s(k+1)$ to be the element of this form with maximal $\alpha'$. Such an element will exist unless $i = n-1$. Comparing the definition of the sequence $\beta(0) , \beta(1) , \dots , \beta(n-1)$ with (\ref{eq:msequence}), we see that $\beta(j) = m_j$. An example of this construction is illustrated in Figure \ref{fig:sequence}.

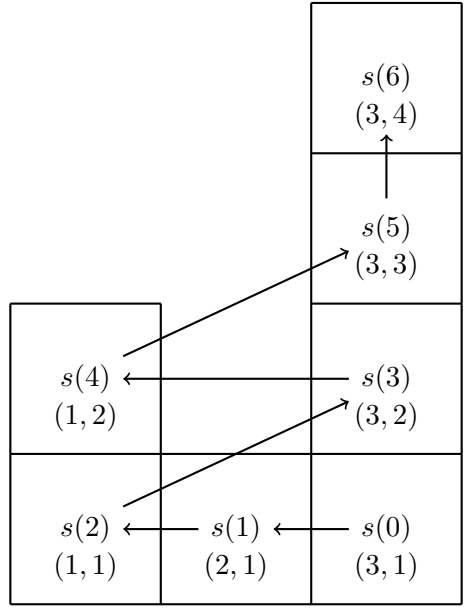
\begin{figure}

\begin{tikzpicture}

\draw[thick] (0,0) -- (6,0);
\draw[thick] (0,2) -- (6,2);
\draw[thick] (0,4) -- (2,4);
\draw[thick] (4,4) -- (6,4);
\draw[thick] (4,6) -- (6,6);
\draw[thick] (4,8) -- (6,8);

\draw[thick] (0,0) -- (0,4);
\draw[thick] (2,0) -- (2,4);
\draw[thick] (4,0) -- (4,8);
\draw[thick] (6,0) -- (6,8);

\node at (5,1) {$s(0)$};
\node at (3,1) {$s(1)$};
\node at (1,1) {$s(2)$};
\node at (5,3) {$s(3)$};
\node at (1,3) {$s(4)$};
\node at (5,5) {$s(5)$};
\node at (5,7) {$s(6)$};

\draw[thick,->] (4.5,1)--(3.5,1);
\draw[thick,->] (2.5,1)--(1.5,1);
\draw[thick,->] (1.5,1.3)--(4.5,2.7);
\draw[thick,->] (4.5,3)--(1.5,3);
\draw[thick,->] (1.5,3.3)--(4.5,4.7);
\draw[thick,->] (5,5.4)--(5,6.25);

\node at (5,0.5) {$(3,1)$};
\node at (3,0.5) {$(2,1)$};
\node at (1,0.5) {$(1,1)$};
\node at (5,2.5) {$(3,2)$};
\node at (1,2.5) {$(1,2)$};
\node at (5,4.5) {$(3,3)$};
\node at (5,6.5) {$(3,4)$};

\end{tikzpicture}
\caption{Construction of the sequence $s(0),s(1), \dots $, for $(n_1,n_2,n_3) = (2,1,4)$.}\label{fig:sequence}
\end{figure}

We say that $j\ge 1$ is a {\em decrease} if $s(j-1) = (a,b)$, $s(j) = (a',b')$ and $a' < a$. For example, in Figure \ref{fig:sequence}, the decreases are $j=1,2,4$. It follows that $(m_j-1)$ is the number of $i \in \{1,2, \dots , j\}$ such that $i$ is not a decrease. Thus for any given $j$, the number of decreases that are less than or equal to $j$ is exactly $j+1-m_j$. For $j=1, \dots , n-1$, let $\epsilon_j = 1$ if $j$ is a decrease and $0$ otherwise. So $\epsilon_1 + \dots + \epsilon_{j} = j+1-m_j$. Now define $E \in \fg(\cK)$ to be:
\begin{equation*}
E = t^{-\epsilon_1} e^*_{s(0)} \otimes e_{s(1)} + t^{-\epsilon_2} e^*_{s(1)} \otimes e_{s(2)} + \dots + t^{-\epsilon_{n-1}} e^*_{s(n-2)} \otimes e_{s(n-1)}.
\end{equation*}
Suppose we are given $f_0 , f_1,f_2, \dots , f_{n-1} \in \mathcal{O}$. We construct a variant of the usual notion of companion matrix, defined as:
\begin{equation*}
A(f_0,f_1, \cdots , f_{n-1}) := E - \sum_{j=0}^{n-1} f_j e^*_{s(j)} \otimes e_{s(0)} \in \tp^\perp.
\end{equation*}
In matrix form $A(f_0 , \cdots , f_{n-1})$ is given by:
\[
A(f_0,\cdots, f_{n-1}) =\begin{bmatrix} 
-f_0 & -f_1 & -f_2  & \cdots & -f_{n-1}  \\
t^{-\epsilon_1} & 0 & 0  & \cdots &   0 \\
0 & t^{-\epsilon_2} & 0 & 0 & 0 \\
\vdots & \vdots & \vdots & \vdots & \vdots \\
0 & \cdots & 0 & t^{-\epsilon_{n-1}} & 0 
\end{bmatrix}.
\]

We claim that $c_j(A) = t^{-j+m_{j-1}}f_{j-1}$. Indeed, the only basis element of $\wedge^{j} V$ which contributes to the trace of $(-1)^j\wedge^{j} A$ is $e_{s(0)} \wedge e_{s(1)} \wedge \dots \wedge e_{s(j-1)}$, and for this element we find
\begin{equation*}
\begin{aligned}
(-1)^j\left(\wedge^{j}\right) A ( e_{s(0)} \wedge e_{s(1)} \wedge \dots \wedge e_{s(j-1)} ) &= (-1)^{j-1} t^{-(\epsilon_1 + \dots + \epsilon_{j-1}) }f_{j-1} e_{s(1)} \wedge e_{s(2)} \wedge \dots \wedge e_{s(j-1)} \wedge e_{s(0)} \\
&= t^{-j+m_{j-1}} f_{j-1} e_{s(0)} \wedge e_{s(1)} \wedge \dots \wedge e_{s(j-1)}.
\end{aligned}
\end{equation*}
Thus $c_j( tA) = t^{m_{j-1}}f_{j-1}$. This proves the inclusion $\bigoplus_{i=0}^{n-1} t^{m_i}\cO \subseteq \chi(t\tp^\perp)$, since $f_0 , f_1 , \dots , f_{n-1} \in \mathcal{O}$ were arbitrary.

\subsection{The inclusion $\chi(t\tp^\perp) \subseteq \bigoplus_{i=0}^{n-1} t^{m_i}\cO$.}

\begin{prop}\label{p:vanishing}
Let $x_1, \dots , x_{j} \in \End(V)$. Then $c_j( x_1 , x_2 , \dots , x_j) = 0$ whenever $j+1-m_{j-1}$ or more of the $x_i$'s belong to $\mathfrak{n}$.
\end{prop}

We will prove the proposition below. But first, let us see why it implies our desired result. 

\begin{cor}
We have the inclusion $\chi(t\tp^\perp) \subseteq \bigoplus_{i=0}^{n-1} t^{m_i}\cO$.
\end{cor}
\begin{proof}
Any $A \in t\tp^\perp$ can be written as $A = y + tz$, where $y \in \fn$ and $z \in \fg(\cO)$. Then $c_j(A) = c_j(y+tz , y+tz , \dots , y+tz)$. Expanding this out and using Proposition \ref{p:vanishing}, we see that $c_j(A)$ is divisible by $t^{m_{j-1}}$. 
\end{proof}

For $1 \le r \le n-1$, we define
\begin{equation*}
J_r = \{ (j_1 , \dots , j_k) \in \mathbb{Z}^k \; | \; 0 \le j_a \le n_a, \; j_1 + \dots + j_k = r \}.
\end{equation*}
For each $(j_1 , \dots , j_k) \in J_{r}$ we further define
\begin{equation*}
S(j_1 , \dots , j_k) = \{ (a,b) \in \mathbb{Z}^2 \; | \; 1 \le a \le k, \; 1 \le b \le j_a \}.
\end{equation*}

Let $\Perm(S(j_1 , \dots , j_k))$ be the set of all permutations of $S(j_1 , \dots , j_k)$. If $\sigma \in \Perm(S(j_1 , \dots , j_k))$ and $(a,b) \in S(j_1 , \dots , j_k)$ we say that $(a,b)$ is a {\em decrease} for $\sigma$ if $\sigma(a,b) = (a',b')$ with $a' < a$. Let $\Dec(\sigma)$ be the total number of decreases of $\sigma$. Let
\begin{equation*}
\begin{aligned}
\lambda(j_1 , \dots , j_k) &= \max_{\sigma \in \Perm(S(j_1 , \dots , j_k))} \left\{ \Dec(\sigma) \right\}, \\
\mu_r &= \max_{(j_1, \dots , j_m) \in J_r} \left\{ \lambda(j_1 , \dots , j_k) \right\}.
\end{aligned}
\end{equation*}

\begin{lem}\label{l:lambda}
Let $(j_1 , \dots , j_k) \in J_r$. Then
\begin{equation*}
\lambda(j_1 , \dots , j_k ) = r - \max (j_1 , \dots , j_k ).
\end{equation*}
\end{lem}
\begin{proof}
We first show that there is a permutation $\sigma$ with $\Dec(\sigma) = r - \max (j_1 , \dots , j_k )$. For this we construct a map $s : \{ 1 , \dots , r \} \to S(j_1 , \dots , j_k)$ as follows. Let $s(1) = (a,1)$, where $a$ is the maximal element of $\{ 1 , \dots , k \}$ with $j_a \neq 0$. Then we obtain $s(i+1)$ from $s(i)$ as follows. Suppose that $s(i) = (\alpha(i) , \beta(i) )$. If there is an element of $S(j_1,\dots , j_k)$ of the form $(\alpha' , \beta(i) )$ for $\alpha' < \alpha$ then we let $s(i+1)$ be the element of this form which has maximal $\alpha'$. If there is no such element then we look for elements of $S(j_1,\dots , j_k)$ of the form $(\alpha' , \beta(i)+1)$. If there is such an element we take $s(i+1)$ to be the element of this form with maximal $\alpha'$. Such an element will exist unless $i = r$. Let $\sigma$ be the cyclic permutation $\sigma = ( s(1) , s(2) , \dots , s(r) )$. From this construction it is clear that the number of decreases of $\sigma$ is $r - \max (j_1 , \dots , j_k )$.\\

Now we show that for any permutation $\sigma$ of $S(j_1 , \dots , j_k)$ we have $\Dec(\sigma) \le r - \max (j_1 , \dots , j_k )$. We claim that if $\sigma$ is not a cyclic permutation then we can find a cyclic permutation $\hat{\sigma}$ with $Dec(\sigma) \le Dec(\hat{\sigma})$. To see this, suppose that $\sigma$ is not cyclic and consider any two cycles of $\sigma$, say, $( s_1 , s_2 , \dots , s_u ) , (t_1 , t_2 , \dots , t_v)$. We can write theses cycles in such a way that $s_u$ and  $t_v$ are not decreases (since every cycle must have at least one non-decrease). Now let $\sigma'$ be the permutation obtained from $\sigma$ by replacing the cycles $( s_1 , s_2 , \dots , s_u ) , (t_1 , t_2 , \dots , t_v)$ with a single cycle $( s_1 , s_2 , \dots , s_u  , t_1 , t_2 , \dots , t_v)$. By construction $\Dec(\sigma) \le \Dec(\sigma')$. Continuing in this fashion we obtain the desired cyclic permutation $\hat{\sigma}$.\\

We have reduced the problem to showing that $\Dec(\sigma) \le r - \max (j_1 , \dots , j_k )$ for any cyclic permutation $\sigma$. Let $a$ be such that $j_a = \max (j_1 , \dots , j_k )$. We consider the entries in $\sigma$ of the form $(a , b)$ for some $b$. That is, we write $\sigma$ as:
\begin{equation*}
\sigma = ( \dots , (a, b_1 ) , \dots , (a , b_2 ) , \dots , (a , b_{j_a} ) , \dots ).
\end{equation*}
Between any two terms $(a , b_i ) , (a , b_{i+1} )$ and between $(a , b_{j_a} )$ and $(a,b_1)$, there must be at least one non-decrease. Thus $\Dec(\sigma) \le r - j_a$ as claimed.
\end{proof}

\begin{lem}\label{l:mu}
We have $\mu_j = j - m_{j-1}$.
\end{lem}
\begin{proof}
By Lemma \ref{l:lambda} we have
\begin{equation*}
\begin{aligned}
\mu_{r} &= \max_{(j_1, \dots , j_k) \in J_{r}} \left\{ \lambda(j_1 , \dots , j_k) \right\} \\
& = \max_{(j_1, \dots , j_k) \in J_{r}} \left\{ r - \max(j_1,\dots , j_k) \right\} \\
&= r - \min_{(j_1, \dots , j_k) \in J_{r}} \left\{ \max(j_1 , \dots , j_k) \right\}.
\end{aligned}
\end{equation*}
Therefore it suffices to show that:
\begin{equation*}
\begin{aligned}
m_{r-1} &= \min_{(j_1, \dots , j_k) \in J_{r}} \left\{ \max(j_1 , \dots , j_k) \right\} \\
&= \min_{\substack{ (j_1, \dots , j_k) \in \mathbb{Z}^k \\ 0 \le j_a \le n_a \\ j_1 + \dots + j_k = r }}\left\{ \max(j_1 , \dots , j_k) \right\}
\end{aligned}
\end{equation*}
But this is immediate, since the $m_j$ are given by the sequence (\ref{eq:msequence}).
\end{proof}

\begin{proof}[Proof of Proposition \ref{p:vanishing}]
For this we may as well take $x_1 , \dots , x_j$ to be of the form $x_u = e^*_{i_u , i'_u} \otimes e_{j_u , j'_u}$ for some pair of sequences $(i_1,i'_1) , \dots , (i_{j} , i'_{j})$ and $(j_1,j'_1) , \dots , (j_{j} , j'_{j})$. To prove Proposition \ref{p:vanishing}, we show that if $c_j( x_1 , \dots , x_{j} ) \neq 0$, then at most $j-m_{j-1}$ of the $x_u$ belong to $\mathfrak{n}$. So suppose that $c_j( x_1 , \dots , x_{j} ) \neq 0$. By Lemma \ref{l:cycles}, we may assume that $(i_1,i'_1) , \dots , (i_{j} , i'_{j})$ are distinct and that $(j_u , j'_u) = (i_{\sigma(u)} , i'_{\sigma(u)})$ for some permutation $\sigma$ of $1, \dots , j$. After re-ordering the basis elements of $V_1, \dots , V_k$, we may as well assume that our sequence $\{(i_u,i'_u)\}$ is of the form 
\begin{equation*}
(1,1),(1,2),\dots , (1,j_1) , (2,1) , \dots , (2,j_2) , \dots , (k , 1) , \dots , (k , j_k),
\end{equation*}
for some $(j_1 , \dots , j_k) \in J_j$. This sequence is exactly the elements of the set $S(j_1 , \dots , j_k)$. Then $\sigma$ is a permutation of $S(j_1 , \dots , j_k)$ and the number of $x_u$ belonging to $\mathfrak{n}$ is precisely the number of decreases of $\sigma$. The proposition now follows, since $\Dec(\sigma) \le \lambda(j_1 , \dots , j_k) \le \mu_j = j-m_{j-1}$, by Lemma \ref{l:mu}.
\end{proof}


\section{Local results for Types $B$, $C$, $D$, and $G_2$} \label{s:Types}
In this section, we prove of the inclusion $\subseteq$ of Theorem \ref{t:main} for Types $B$, $C$, $D$, and $G_2$.

\subsection{Type $B$ }

Let $\fg = \mathfrak{so}_{2n+1}$ which we view as the subalgebra of $\Sl_{2n+1}$ preserving a non-degenerate symmetric bilinear form. Let $V$ be the standard $2n+1$-dimensional representation. A parabolic subalgebra $\fp \subseteq \mathfrak{so}_{2n+1}$ can be described as the subalgebra of $\fg$ preserving a flag of isotropic subspaces
\begin{equation*}
\{ 0 \} = F_0 \subset F_1 \subset F_2 \subset \dots \subset F_k.
\end{equation*}
Clearly any $A \in \mathfrak{so}_{2n+1}$ preserving this flag also preserves the flag of length $2k+1$ given by
\begin{equation}\label{eq:flag2n+1}
\{ 0 \} = F_0 \subset F_1 \subset F_2 \subset \dots \subset F_k \subset F_k^\perp \subset F_{k-1}^\perp \subset \dots \subset F_{1}^\perp \subset F_0^\perp = V.
\end{equation}
Note that $F_k \neq F_k^\perp$ because $\dim(V)$ is odd. Let $r = \dim(F_k)$. Then $W := F_k^\perp/F_{k-1}$ inherits an orthogonal structure from $V$. Let $2s+1 = \dim(F_k^\perp/F_k)$ so that $2n+1 = 2r+2s+1$, or $n = r+s$. For $j = 1, \dots , k$ let $r_j = \dim(F_j/F_{j-1})$ so that $r = r_1 + r_2 + \dots + r_k$ is a partition of $r$. We may choose splittings $F_j = F_{j-1} \oplus V_j$, $F_k^\perp = F_k \oplus W$ and $F_{j-1}^\perp = F_j^\perp \oplus V_j^*$ so that
\begin{equation*}
V = V_1 \oplus V_2 \oplus \dots \oplus V_k \oplus W \oplus V_k^* \oplus \dots \oplus V_2^* \oplus V_1^*
\end{equation*}
and such that the orthogonal pairing pairs $V_j$ with $V_j^*$ and $W$ with itself. 

The flag of subspaces in $V$ given by (\ref{eq:flag2n+1}) defines a parabolic subalgebra $\tilde{\fp} \subseteq \Sl_{2n+1}$ for which $\fp = \tilde{\fp} \cap \mathfrak{so}_{2n+1}$. Let $\fn$ denote the nilpotent radical of $\fp \subseteq \mathfrak{so}_{2n+1}$ and similarly let $\tilde{\fn}$ be the nilpotent radical of $\tilde{\fp} \subseteq \Sl_{2n+1}$. It follows that $\fn = \tilde{\fn} \cap \mathfrak{so}_{2n+1}$. In particular, we have an inclusion
\begin{equation*}
\tp^\perp=t^{-1}\fn +\fg(\cO) \subseteq t^{-1}\tilde{\fn}  + \Sl_{2n+1}(\cO).
\end{equation*}
We will use this to obtain lower bounds for the valuations of the invariant polynomials $c_2 , c_4 , \dots , c_{2n}$ from the corresponding lower bounds obtained for type A in the previous section. Let $\tilde{\fl}$ be the Levi subalgebra of $\tilde{\fp}$ and $\tilde{m}_1 \le \tilde{m}_2 \le \dots \le \tilde{m}_{2n}$ the fundamental degrees of $\tilde{\fp}$ arranged in non-decreasing order. Then since we have proven Theorem \ref{t:main} for type A, we immediately have:
\begin{equation*}
\chi(\tp^\perp) \subseteq \bigoplus_{i=1}^{n} t^{-2i+\tilde{m}_{2i-1}}\cO.
\end{equation*}
Therefore, to prove the result for type $B$ it suffices to show the following:
\begin{prop}
We have an equality $m_i = \tilde{m}_{2i-1}$.
\end{prop}
\begin{proof}
Let $\delta_1 \ge \delta_2 \ge \dots \ge \delta_m$ be the conjugate partition to $r = r_1 + r_2 + \dots + r_k$. There are four cases to consider according to whether $m$ is even or odd and whether $m \le 2s$ or $m > 2s$. We give the proof in the case that $m$ is even and $m \le 2s$, the other cases being similar. The flag given by (\ref{eq:flag2n+1}) corresponds to the partition of $2n+1$ given by $(r_1, r_1,r_2,r_2, \dots , r_k, r_k , 2s+1)$. Since $m$ is even and $m \le 2s$ it follows that the sequence $\tilde{m}_0 , \tilde{m}_1 , \dots , \tilde{m}_{2n}$ (where we set $\tilde{m}_0 = 1$) is given by:
\begin{equation*}
\underbrace{1,1, \dots , 1}_{2\delta_1+1 \text{ times}}, \underbrace{2,2, \dots , 2}_{2\delta_2+1 \text{ times}}, \dots , \underbrace{m,m, \dots , m}_{2\delta_{m}+1 \text{ times}} , m+1 , m+2 , \dots , 2s , 2s+1.
\end{equation*}
Therefore, the subsequence $\tilde{m}_1 , \tilde{m}_3 , \dots , \tilde{m}_{2n-1}$ is given by:
\begin{equation*}
\underbrace{1,1, \dots , 1}_{\delta_1 \text{ times}}, \underbrace{2,2, \dots , 2}_{\delta_2+1 \text{ times}}, \underbrace{3,3, \dots , 3}_{\delta_3 \text{ times}}, \underbrace{4,4, \dots , 4}_{\delta_4+1 \text{ times}}, \dots , \underbrace{m,m, \dots , m}_{\delta_{m}+1 \text{ times}} , m+2 , m+4 , \dots , 2s.
\end{equation*}

On the other hand, the Levi subalgebra $\fl \subseteq \fp$ is given by
\begin{equation*}
\fl \cong \bigoplus_{i} End(V_i) \oplus \mathfrak{so}(W).
\end{equation*}
Thus the fundamental degrees of $\fl$ are $1,2, \dots , r_1 , 1 , 2 , \dots , r_2 , \dots , 1 , 2 , \dots , r_k , 2 , 4 , 6 , \dots , 2s$. The sequence $m_1 \le m_2 \le \dots \le m_n$ is obtained by putting the fundamental degrees in non-decreasing order. Clearly this is the same as the sequence $\tilde{m}_1 , \tilde{m}_3 , \dots , \tilde{m}_{2n-1}$ given above.
\end{proof}

\subsection{Type $C$} 

We proceed in much the same way as for type $B$. Let $\fg = \mathfrak{sp}_{2n}$ which we view as the subalgebra of $\Sl_{2n}$ preserving a symplectic form. Let $V$ be the standard $2n$-dimensional representation. A parabolic subalgebra $\fp \subseteq \mathfrak{sp}_{2n}$ can be described as the subalgebra of $\fg$ preserving a flag of isotropic subspaces
\begin{equation*}
\{ 0 \} = F_0 \subset F_1 \subset F_2 \subset \dots \subset F_k.
\end{equation*}
Any $A \in \mathfrak{sp}_{2n}$ preserving this flag also preserves the flag given by
\begin{equation}\label{eq:flag2n}
\{ 0 \} = F_0 \subset F_1 \subset F_2 \subset \dots \subset F_k \subseteq F_k^\perp \subset F_{k-1}^\perp \subset \dots \subset F_{1}^\perp \subset F_0^\perp = V.
\end{equation}
In the type $C$ case, the inclusion $F_k \subseteq F_k^\perp$ may be an equality (this happens precisely when $F_k$ is a Lagrangian). In any case the (possibly zero-dimensional) space $W := F_k^\perp/F_{k-1}$ inherits a symplectic structure from $V$. Let $2s = \dim(F_k^\perp/F_k)$ so that $2n = 2r+2s$ and note that $s$ may be zero. For $j = 1, \dots , k$ let $r_j = \dim(F_j/F_{j-1})$ so that $r = r_1 + r_2 + \dots + r_k$ is a partition of $r$. We may choose splittings $F_j = F_{j-1} \oplus V_j$, $F_k^\perp = F_k \oplus W$ and $F_{j-1}^\perp = F_j^\perp \oplus V_j^*$ so that
\begin{equation*}
V = V_1 \oplus V_2 \oplus \dots \oplus V_k \oplus W \oplus V_k^* \oplus \dots \oplus V_2^* \oplus V_1^*
\end{equation*}
and such that the symplectic form pairs $V_j$ with $V_j^*$ and $W$ with itself. 

The flag of subspaces in $V$ given by (\ref{eq:flag2n}) defines a parabolic subalgebra $\tilde{\fp} \subseteq \Sl_{2n}$ for which $\fp = \tilde{\fp} \cap \mathfrak{sp}_{2n}$. Let $\fn$ denote the nilpotent radical of $\fp \subseteq \mathfrak{sp}_{2n}$ and similarly define $\tilde{\fn}$ as the nilpotent radical of $\tilde{\fp} \subseteq \Sl_{2n}$. It follows that $\fn = \tilde{\fn} \cap \mathfrak{sp}_{2n}$ and we obtain an inclusion
\begin{equation*}
\tp^\perp=t^{-1}\fn +\fg(\cO) \subseteq t^{-1}\tilde{\fn} + \Sl_{2n}(\cO).
\end{equation*}
Let $\tilde{\fl}$ be the Levi subalgebra of $\tilde{\fp}$ and $\tilde{m}_1 \le \tilde{m}_2 \le \dots \le \tilde{m}_{2n}$ the fundamental degrees of $\tilde{\fp}$ arranged in non-decreasing order. As we have proven Theorem \ref{t:main} for type A, we immediately have:
\begin{equation*}
\chi(\tp^\perp) \subseteq \bigoplus_{i=1}^{n} t^{-2i+\tilde{m}_{2i-1}}\cO.
\end{equation*}
To prove the result for type $C$ it suffices to show the following:
\begin{prop}
We have an equality $m_i = \tilde{m}_{2i-1}$.
\end{prop}
\begin{proof}
The proof is almost identical to the type $B$ case. Let $\delta_1 \ge \delta_2 \ge \dots \ge \delta_m$ be the conjugate partition to $r = r_1 + r_2 + \dots + r_k$. There are four cases to consider according to whether $m$ is even or odd and whether $m \le 2s$ or $m > 2s$. We give the proof in the case that $m$ is even and $m \le 2s$. The flag given by (\ref{eq:flag2n}) corresponds to the partition of $2n$ given by $(r_1, r_1,r_2,r_2, \dots , r_k, r_k , 2s)$. Since $m$ is even and $m \le 2s$ it follows that the sequence $\tilde{m}_0 , \tilde{m}_1 , \dots , \tilde{m}_{2n}$ (where we set $\tilde{m}_0 = 1$) is given by:
\begin{equation*}
\underbrace{1,1, \dots , 1}_{2\delta_1+1 \text{ times}}, \underbrace{2,2, \dots , 2}_{2\delta_2+1 \text{ times}}, \dots , \underbrace{m,m, \dots , m}_{2\delta_{m}+1 \text{ times}} , m+1 , m+2 , \dots , 2s.
\end{equation*}
Therefore, the subsequence $\tilde{m}_1 , \tilde{m}_3 , \dots , \tilde{m}_{2n-1}$ is given by:
\begin{equation*}
\underbrace{1,1, \dots , 1}_{\delta_1 \text{ times}}, \underbrace{2,2, \dots , 2}_{\delta_2+1 \text{ times}}, \underbrace{3,3, \dots , 3}_{\delta_3 \text{ times}}, \underbrace{4,4, \dots , 4}_{\delta_4+1 \text{ times}}, \dots , \underbrace{m,m, \dots , m}_{\delta_{m}+1 \text{ times}} , m+2 , m+4 , \dots , 2s.
\end{equation*}

On the other hand, the Levi subalgebra $\fl \subseteq \fp$ is given by
\begin{equation*}
\fl \cong \bigoplus_{i} End(V_i) \oplus \mathfrak{sp}(W).
\end{equation*}
Thus the fundamental degrees of $\fl$ are $1,2, \dots , r_1 , 1 , 2 , \dots , r_2 , \dots , 1 , 2 , \dots , r_k , 2 , 4 , 6 , \dots , 2s$ (if $s=0$ the fundamental degrees are $1,2, \dots , r_1 , 1 , 2 , \dots , r_2 , \dots , 1 , 2 , \dots , r_k$). The sequence $m_1 \le m_2 \le \dots \le m_n$ is obtained by putting the fundamental degrees in non-decreasing order. Clearly this is the same as the sequence $\tilde{m}_1 , \tilde{m}_3 , \dots , \tilde{m}_{2n-1}$ given above.
\end{proof}

\subsection{Type $D$ under Assumption \ref{ass:main2}}\label{ss:typeD}

For a parabolic in type $D$ satisfying Assumption \ref{ass:main}, we can prove the results in much the same way as we have done for types $B$ and $C$. We omit the proof as the details are very similar to these two cases. 

\subsection{Type $G_2$} 

Let $\fg$ be the Lie algebra of type $G_2$ and $V$ its $7$-dimensional representation. The $G_2$ Lie algebra preserves a symmetric non-degenerate bilinear form $\langle \; , \; \rangle$ and a $3$-form $\varphi \in \wedge^3 V^*$. There are three parabolic subalgebras to consider, two maximal parabolics and the Borel subalgebra. We consider each of these separately.

\subsubsection{Stabiliser of an isotropic line} Let $F_1 \subset V$ be a $1$-dimensional isotropic subspace. Let $\fp \subset \fg$ be the maximal parabolic subalgebra preserving $F_1$. Suppose that $F_1$ is spanned by $v_1$. Set $F_3 = \{ v \in V \; | \; \varphi(v,v_1 , \; . \;  ) = 0 \}$. One can check that $F_3$ is a $3$-dimensional isotropic subspace of $V$ containing $F_1$, hence $\fp$ preserves the flag given by
\begin{equation}\label{eq:g2flag1}
\{0\} \subset F_1 \subset F_3 \subset F_3^\perp \subset F_1^\perp \subset V.
\end{equation}
Thus $\fp$ is contained in the parabolic subalgebra $\tilde{\fp}$ of $\Sl_7$ preserving this flag. Letting $\fn$ denote the nilpotent radical of $\fp$ and $\tilde{\fn}$ the nilpotent radical of $\tilde{\fp}$, we have an inclusion $\fn \subset \tilde{\fn}$. Hence we also have an inclusion $\tp^\perp=t^{-1}\fn +\fg(\cO) \subset t^{-1}\tilde{\fn} + \Sl_{7}(\cO)$. Let $\tilde{\fl}$ be the Levi subalgebra of $\tilde{\fp}$ and $\tilde{m}_1 \le \tilde{m}_2 \le \dots \le \tilde{m}_{6}$ the fundamental degrees of $\tilde{\fp}$ arranged in non-decreasing order. The flag (\ref{eq:g2flag1}) corresponds to the partition $7 = 1 + 2 + 1 + 1 + 2 + 1$, so we get $\tilde{m}_1 , \dots , \tilde{m}_6 = 1 , 1 , 1 , 1 , 1 , 2 , 2$. This gives an inclusion
\begin{equation*}
\chi(t \tp^\perp) \subseteq t^{\tilde{m}_1} \cO \oplus t^{\tilde{m}_5} \cO = t \cO \oplus t^2 \cO.
\end{equation*}
On the other hand, the Levi subalgebra of $\fp$ is isomorphic to $\Gl_2$, which has fundamental degrees $m_1 , m_2 = 1 , 2$. Thus $\chi(t \tp^\perp) \subseteq t^{m_1}\cO \oplus t^{m_2} \cO$ as required.

\subsubsection{Stabiliser of a $\varphi$-isotropic plane} Let $F_2 \subset V$ be a $2$-dimensional subspace of $V$. We say $F_2$ is $\varphi$-isotropic if $\varphi( v , w , \; . \; ) = 0$ for all $v,w \in F_2$. One can check that a $\varphi$-isotropic plane is also isotropic with respect to $\langle \; , \; \rangle$. Let $\fp \subset \fg$ be the subalgebra preserving a $\varphi$-isotropic plane $F_2$. Then $\fp$ is a maximal parabolic subalgebra. Clearly $\fp$ preserves the flag 
\begin{equation}\label{eq:g2flag2}
\{ 0 \} \subset F_2 \subset F_2^\perp \subset V.
\end{equation}
Let $\tilde{\fp}$ be the parabolic subalgebra of $\Sl_7$ preserving this flag. Letting $\fn$ denote the nilpotent radical of $\fp$ and $\tilde{\fn}$ the nilpotent radical of $\tilde{\fp}$, we have as usual inclusions $\fn \subset \tilde{\fn}$ and $\tp^\perp=t^{-1}\fn +\fg(\cO) \subset t^{-1}\tilde{\fn} + \Sl_{7}(\cO)$. Let $\tilde{\fl}$ be the Levi subalgebra of $\tilde{\fp}$ and $\tilde{m}_1 \le \tilde{m}_2 \le \dots \le \tilde{m}_{6}$ the fundamental degrees of $\tilde{\fp}$ arranged in non-decreasing order. The flag (\ref{eq:g2flag2}) corresponds to the partition $7 = 2 + 3 + 2$, so we get $\tilde{m}_1 , \dots , \tilde{m}_6 = 1 , 1 , 1 , 2 , 2 , 2 , 3$. This gives an inclusion
\begin{equation*}
\chi(t \tp^\perp) \subseteq t^{\tilde{m}_1} \cO \oplus t^{\tilde{m}_5} \cO = t \cO \oplus t^2 \cO.
\end{equation*}
The Levi subalgebra of $\fp$ is isomorphic to $\Gl_2$, which has fundamental degrees $m_1 , m_2 = 1 , 2$. Thus $\chi(t \tp^\perp) \subseteq t^{m_1}\cO \oplus t^{m_2} \cO$ as required.

\subsubsection{Borel subalgebra}

Suppose that $\fp \subset \fg$ is a Borel subalgebra and $\fn^-$ the opposite nilpotent radical. In this case $\fp$ will preserve a flag $F_1 \subset F_2$, where $F_1$ is an isotropic line and $F_2$ is a $\varphi$-isotropic plane. From (\ref{eq:g2flag1}) and (\ref{eq:g2flag2}) we see that $\fp$ preserves a maximal flag, which corresponds to the partition $7 = 1 + 1 + 1 + 1 + 1 + 1 + 1$. Arguing as in the previous cases we get an inclusion
\begin{equation*}
\chi(t \tp^\perp) \subseteq t \cO \oplus t \cO.
\end{equation*}
The Levi subalgebra of $\fp$ is isomorphic to $\Gl_1 \times \Gl_1$, which has fundamental degrees $m_1 = 1$, $m_2 = 1$. Thus $\chi(t \tp^\perp) \subseteq t^{m_1}\cO \oplus t^{m_2} \cO$ as required.


\section{Local results for Type $D$}\label{s:TypeD}

In this section we consider a parabolic in type $D$ which {\em need not satisfy Assumption \ref{ass:main2}}. Determining the base in this case presents new difficulties not seen in the other types. As such our approach will be quite different to the previous cases. On other hand, if the parabolic does satisfy Assumption \ref{ass:main2}, then it is easy to see that the description provided for the parabolic Hitchin base in this section coincides with the description given in Theorem \ref{t:main}. 

Let $\fg = \mathfrak{so}_{2n}$ and let $\fp \subseteq \fg$ be a parabolic subalgebra. Suppose that the Richardson nilpotent orbit associated to $\fp$ has Jordan blocks of sizes $\delta_1 \ge \delta_2 \ge \dots \ge \delta_{\mu}$. These have been computed in \cite{Spaltenstein1}, see also \cite{Duckworth}. Recall \cite[Chapter 5]{CM} that for each even number $2a$, the partition $(\delta_1 , \delta_2 , \dots , \delta_\mu)$ has an even number of parts of size $2a$. In particular it follows that $\mu$ is even. Moreover, for a Richardson nilpotent in $\mathfrak{so}_{2n}$, it follows from \cite{Spaltenstein1, Duckworth} that $\delta_{2j-1}$ and $\delta_{2j}$ are either both even or both odd, so even and odd parts occur in consecutive pairs.\\

To describe the base in type $D$, it will be convenient to make use of Newton polygons. Suppose that $p(\lambda , t) \in \cO[\lambda]$ is a monic polynomial of degree $2n$ in $\lambda$ with coefficients in $\cO = \mathbb{C}[\![t]\!]$. Write $p(\lambda , t)$ as:
\begin{equation*}
p(\lambda , t) = \lambda^{2n} + f_1 \lambda^{2n-1} + f_2 \lambda^{2n-2} + \dots + f_{2n} = \lambda^{2n} + \sum_{\substack{ 0 \le \beta < 2n \\ \alpha \ge 0} }\rho_{\alpha , \beta} \lambda^\beta t^\alpha,
\end{equation*}
where $f_1 , \dots , f_{2n} \in \cO$ and $\rho_{\alpha , \beta} \in \mathbb{C}$. 

\begin{defe} We say that $p(\lambda , t)$ lies in the Newton polygon of slopes $-\delta_1 , -\delta_2 , \dots , -\delta_{\mu}$ if for all $0 \le \beta < 2n$, we have that $\rho_{\alpha , \beta} \neq 0$ only if $\alpha \ge j$, where $j$ is the unique integer such that:
\begin{equation*}
\delta_1 + \dots + \delta_{j-1} < 2n- \beta \le \delta_1 + \dots + \delta_{j}.
\end{equation*}
\end{defe} 

For example, the Newton polygon with slopes $-4,-3,-1,-1$ is shown in Figure \ref{fig:newt}. The polynomial $p(\lambda,t) = \sum_{\alpha , \beta} \rho_{\alpha , \beta} \lambda^\beta t^\alpha$ lies in this Newton polygon provided that $\rho_{\alpha , \beta} \neq 0$ only for values of $(\alpha , \beta)$ contained in the shaded region.

\begin{rem} The polynomial $p=p(\lambda,t)$ defines a convex region $D_p$ of the plane whose boundary is what is known as the Newton polygon of $p(\lambda,t)$; cf.  \cite[II.6]{Neukirch}. On the other hand, the slopes $-\delta_1 , -\delta_2 , \dots , -\delta_{\mu}$ define another convex region in the plane $D'$ and the Newton polygon of $p$ lies in $D'$ (in the above sense) if and only if $D_p\subseteq D'$. 
\end{rem} 

\begin{figure}
\includegraphics{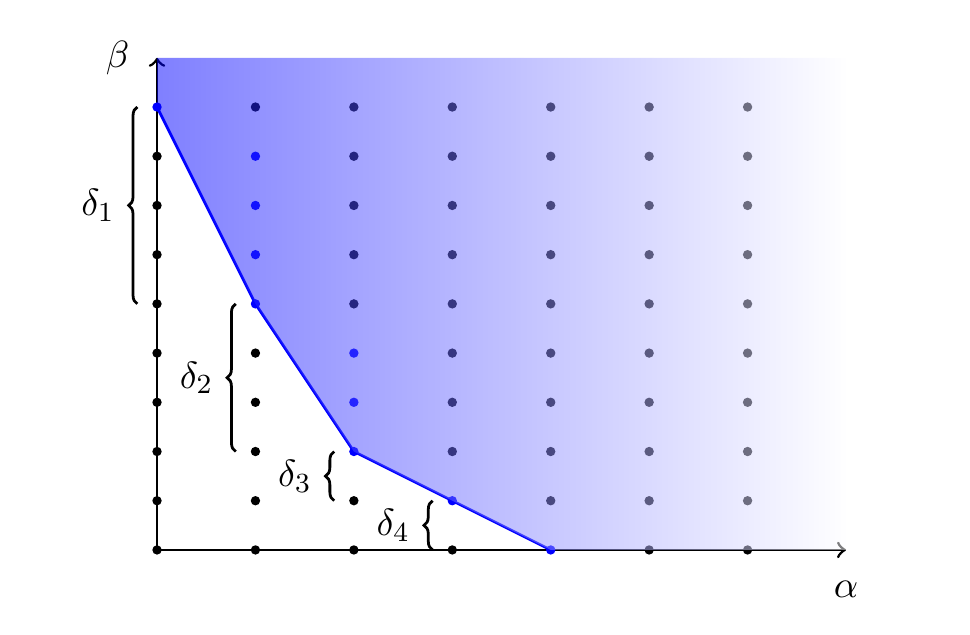}
\caption{Newton polygon for $(\delta_1,\delta_2,\delta_3,\delta_4) = (4,3,1,1)$.}\label{fig:newt}
\end{figure}

\begin{ntn}
Let $\fc:=\bC[\fg]^G$ denote the invariant ring. As mentioned in the introduction, Chevalley's theorem implies that if we choose a basis of invariant polynomials, then we get an isomorphism $\fc\simeq \bA^n$. On the other hand, the notation $\fc$ serves to remind us that this space has a nontrivial $\bGm$ action. We shall use this $\bGm$ action below.  
\end{ntn}

Recall the local Hitchin map $\chi : \fg(\cK) \to \fc(\cK) \cong \cK^n$, where we use the basis $(c_2 , c_4 , \dots , c_{2n-2} , p_n)$ of invariant polynomials described in Assumption \ref{ass:main} to identify $\fc$ with $\mathbb{A}^n$. Similarly $\chi : \fg(\cO) \to \fc(\cO) \cong \cO^n$, in particular $\chi( t \tp^\perp ) \subseteq \fc(\cO)$. Next we let $\cD'  \subseteq \fc(\cO) \cong \cO^n$ be defined as follows:
\begin{defe} Let $\fd'  \subseteq \fc(\cO) \cong \cO^n$ be the subset consisting of points $(c_2 , c_4 , \dots , c_{2n-2} , p_n )$ for which the corresponding characteristic polynomial $\lambda^{2n} + c_2 \lambda^{2n-2} + \dots + c_{2n-2} \lambda^2 + p_n^2$ lies in the Newton polygon of slopes $-\delta_1 , -\delta_2 , \dots , -\delta_{\mu}$.
\end{defe}

Let $\tilde{\delta}_1 > \tilde{\delta}_2 > \dots > \tilde{\delta}_r$ denote the distinct parts of $( \delta_1 , \dots , \delta_{\mu} )$ and let $e_1 , e_2 , \dots , e_r$ be the multiplicities, so $e_j$ is even whenever $\tilde{\delta}_j$ is even. Suppose that $\tilde{\delta}_j$ is even. The boundary of the Newton polygon has an edge segment of slope $-\tilde{\delta}_j$. The integer points $(\alpha , \beta) \in \mathbb{Z}^2$ lying on this segment are given by:
\begin{equation}\label{eq:relevantPairs}
(\alpha , \beta) = (e_1 + \dots +e_{j-1} , 2n - e_1 \tilde{\delta}_1 - \dots -e_{j-1}\tilde{\delta}_{j-1} ) + s( 1 , -\tilde{\delta}_j ), \quad \text{for } s = 0 , 1 , \dots , e_j.
\end{equation}
Let us denote these points as $( \alpha_{j,s} , \beta_{j,s} )$. Now suppose that $(c_2 , \dots , c_{2n-2} , p_n ) \in \fd'$ and write
\begin{equation}\label{eq:Newton}
\lambda^{2n} + c_2 \lambda^{2n-2} + \dots + c_{2n-2} \lambda^2 + p_n^2 = \lambda^{2n} + \sum_{\substack{ 0 \le \beta \leq 2n \\ \alpha \ge 0} }\rho_{\alpha , \beta} \lambda^\beta t^\alpha.
\end{equation}
For each $j$ such that $\tilde{\delta}_j$ is even, define a polynomial $q_j(u) \in \mathbb{C}[u]$ by
\begin{equation*}
q_j(u) = \sum_{s = 0}^{e_j} \rho_{\alpha_{j,s} , \beta_{j,s} } u^{e_j-s}.
\end{equation*}
Since for each $i$, we have that $\delta_{2i-1}, \delta_{2i}$ have the same parity, one sees that that $\alpha_{j,s} , \beta_{j,s}$ are both even whenever $e_j$ is even. 

\begin{defe}\label{def:d}
Let $\fd \subset \fd'$ be the subset consisting of $(c_2 , \dots , c_{2n-2} , p_n) \in \fd'$ satisfying the following property: for each $j$ such that $e_j$ is even we have that $q_j(u)$ is a square, i.e. $q_j(u) = r_j(u)^2$ for some $r_j(u) \in \mathbb{C}[u]$. 
\end{defe}

\begin{rem} \label{r:dimension}
Note that for each $j$, the condition that $q_j(u)$ is a square defines a subvariety of codimension $e_j/2$. Indeed, our condition that $q_j(u)$ is a square means that we require $q_j(u) \in S^{e_j}(\mathbb{C}^2)$ to lie in the image of the affine Veronese map $\mathbb{C}^{e_j/2+1} \cong S^{e_j/2} (\mathbb{C}^2) \to S^{e_j} (\mathbb{C}^2) \cong \mathbb{C}^{e_j + 1}$. The image is an affine subvariety of $\mathbb{C}^{e_j+1}$ of codimension $e_j/2$. 
\end{rem}

For $( c_2 , c_4 , \dots , c_{2n-2} , p_n ) \in \cK^n$, consider the following weighted action of $t^{-1}$:
\[
t^{-1} \bullet ( c_2 , c_4 , \dots , c_{2n-2} , p_n ) = ( t^{-2} c_2 , t^{-4} c_4 , \dots , t^{-(2n-2)} c_{2n-2} , t^{-n} p_n ).
\]
This action is defined so that $\chi( t^{-1}\phi ) = t^{-1} \bullet \chi(\phi)$.

\begin{prop}\label{p:localbaseD}
We have an inclusion $\chi( t \tp^\perp ) \subseteq \fd$ (or equivalently, $\chi( \tp^\perp ) \subseteq t^{-1} \bullet \fd$).
\end{prop}
\begin{proof}
Let $\mathfrak{b} = \fc(\cO) \cong \cO^n$ and let $\chi : t \tp^\perp \to \mathfrak{b}$ be the local Hitchin map. Then $\fd \subseteq \mathfrak{b}$. We will define truncated versions of $\fd , \mathfrak{b} , \chi$ as follows. Let $\mathfrak{b}^{tr} = \fc(\cO)/ \fc(t \cO) = \left( \bigoplus_{j=1}^{n-1} \cO / t^{2j} \cO \right) \oplus \left( \cO/ t^n \cO \right)$ and $\pi : \mathfrak{b} \to \mathfrak{b}^{tr}$ the obvious projection. Let $\fd^{tr} \subseteq \mathfrak{b}^{tr}$ be the image of $\fd$ under $\pi$ and $\chi^{tr} : t \tp^\perp \to \mathfrak{b}^{tr}$ the composition $\chi^{tr} = \pi \circ \chi$. Then it suffices to show $\chi^{tr}( t \tp^\perp ) \subseteq \fd^{tr}$. Note also that $\chi^{tr}$ factors through the projection $t \tp^\perp \to t \tp^\perp /( t^{2n} \fg(\cO))$ and so we can view $\chi^{tr}$ as a map
\begin{equation*}
\chi^{tr} : t \tp^\perp / (t^{2n} \fg(\cO)) \to \mathfrak{b}^{tr}.
\end{equation*}
The point of doing this is that $t \tp^\perp / (t^{2n} \fg(\cO))$ and $\mathfrak{b}^{tr}$ are in a natural way finite dimensional complex affine spaces and $\chi^{tr}$ is easily seen to be a polynomial map between these spaces. It is also clear that $\fd^{tr}$ is an affine subvariety of $\mathfrak{b}^{tr}$, in particular it is Zariski closed. Therefore it is enough to show that $\chi^{tr}( \mathfrak{r} ) \subseteq \fd^{tr}$, where $\mathfrak{r} \subset t \tp^\perp $ are the elements of the form $t \phi$, where $\phi \in \tp^\perp$ has residue a Richardson nilpotent. It also follows that it is sufficient to show $\chi( \mathfrak{r} ) \subseteq \fd$. 

By our results in type $A$ and using the inclusion $\mathfrak{so}_{2n} \subset \mathfrak{sl}_{2n}$, we immediately see that $\chi( \mathfrak{r} ) \subseteq \fd'$. In other words, we have that the image of $\chi$ on elements $t \phi$ where $\phi$ has a residue with Jordan blocks of sizes $\delta_1 \ge \delta_2 \ge \dots \ge \delta_{\mu}$ lies in the corresponding Newton polygon of slopes $-\delta_1, -\delta_2 , \dots , -\delta_{\mu}$. For $t\phi \in \mathfrak{r}$, let $p(\lambda , t) = \det( \lambda - t\phi) \in \cO[\lambda]$ be the characteristic polynomial. From \cite[Proposition 6.4]{Spaltenstein2}, we know that (for generic elements of $\mathfrak{r}$) $p(\lambda , t)$ admits a factorisation of the form
\begin{equation*}
p(\lambda , t) = p^{odd}(\lambda , t) \prod_{ \{j \; | \; e_j \; \text{even} \}} p_j(\lambda , t)p_j(-\lambda , t),
\end{equation*}
where $p^{odd}(\lambda , t)$ is the product of all monic irreducible factors whose Newton polygon has odd slope and $p_j(\lambda , t)p_j(-\lambda , t)$ is the product of all monic irreducible factors whose Newton polygon has slope $e_j$. The fact that this product can be written in the form $p_j(\lambda , t)p_j(-\lambda , t)$ for some $p_j(\lambda , t) \in \cO[\lambda]$ is part of the statement of Proposition 6.4 in \cite{Spaltenstein2}.

We have that $p_j(\lambda , t)$ is a monic, degree $\tilde{\delta}_j e_j/2$ polynomial in $\lambda$ lying on the Newton polygon of slope $-\tilde{\delta}_j$. Thus it has the form
\begin{equation*}
p_j(\lambda , t) = \lambda^{\tilde{\delta}_je_j/2} + y_1 t \lambda^{\tilde{\delta}_j (e_j/2-1)} + y_2 t^2 \lambda^{\tilde{\delta}_j (e_j/2 - 2)} + \dots + y_{e_j/2}t^{e_j/2} 
+ s_j(\lambda,t),
\end{equation*}
where $y_1 , y_2 , \dots , y_{e_j/2} \in \mathbb{C}$ and $s_j(\lambda , t)$ consists only of terms which do not lie on the boundary of the Newton polygon of slope $-\tilde{\delta}_j$. As $\tilde{\delta}_j$ is even, we have:
\begin{equation*}
p_j(-\lambda , t) = \lambda^{\tilde{\delta}_je_j/2} + y_1 t \lambda^{\tilde{\delta}_j (e_j/2-1)} + y_2 t^2 \lambda^{\tilde{\delta}_j (e_j/2 - 2)} + \dots + y_{e_j/2}t^{e_j/2} 
+ s_j(-\lambda,t).
\end{equation*}
It follows that 
\begin{equation*}
q_j(u) = c \left( u^{e_j/2} + y_1 u^{e_j/2 - 1} + \dots + y_{e_j/2} \right)^2
\end{equation*}
for some constant $c$. Hence $\chi( t\phi ) \in \fd$.
\end{proof}

\begin{exam}\label{ex:so}
Here we consider an example of a parabolic in $\fp \subset \mathfrak{so}_{10}$ for which the local truncated base $\fd^{tr}$ is singular. The corresponding global Hitchin base will have corresponding singularities. Let $\fp$ be the parabolic subalgebra in $\mathfrak{so}_{10}$ obtained by crossing off the two end nodes.

\begin{center}
\begin{tikzpicture}
\draw [thick] (2,0) -- (4,0) ;
\draw [thick] (4,0) -- (4.9,0.45);
\draw [thick] (4,0) -- (4.9,-0.45) ;

\draw [fill] (2,0) circle(0.1);
\draw [fill] (3,0) circle(0.1);
\draw [fill] (4,0) circle(0.1);

\draw[thick] (4.95,0.55) -- (5.15 , 0.35);
\draw[thick] (5.15,0.55) -- (4.95 , 0.35);

\draw[thick] (4.95,-0.55) -- (5.15 , -0.35);
\draw[thick] (5.15,-0.55) -- (4.95 , -0.35);

\end{tikzpicture}
\end{center}

From \cite{Duckworth}, the Richardson has Jordan blocks $(\delta_1 , \delta_2 , \delta_3 , \delta_4) = (3,3,2,2)$. Thus $\tilde{\delta}_1 = 3$, $\tilde{\delta}_2 = 2$, $e_1 = e_2 = 2$. Writing out \eqref{eq:Newton}, we get 
\begin{multline}
\lambda^{10} + c_2 \lambda^{8} + c_4 \lambda^6 + c_6\lambda^4+c_4\lambda^6+c_2\lambda^2+c_0=\\
\lambda^{10}+ 
\lambda^6(\sum_{\alpha\geq 0} \rho_{\alpha,6}) t^\alpha) + 
\lambda^8(\sum_{\alpha\geq 0} \rho_{\alpha,8}t^\alpha) +
\lambda^8(\sum_{\alpha\geq 0} \rho_{\alpha,8}t^\alpha) +
\lambda^8(\sum_{\alpha\geq 0} \rho_{\alpha,8}t^\alpha) +
\lambda^8(\sum_{\alpha\geq 0} \rho_{\alpha,8}t^\alpha)
\end{multline}
The relevant pairs specified by \eqref{eq:relevantPairs} are
\[
(\alpha,\beta)=(2,4)+s(1,-2), \,\, s=0,1,2,\implies (\alpha,\beta)=(2,4), (3,2), (4,0). 
\]
Let $x = \rho_{2,4} , y = \rho_{3,2} , z = \rho_{4,0}$. These are the coefficients of the characteristic polynomial lying on the edges of the Newton polygon with slope $-2$. Thus $q_2(u) = xu^2 + yu + z$. We also have that $z = w^2$, where $w$ is the $t^2$-coefficient of the pfaffian $p_n$. We require that $q_2(u)$ is a square, i.e.
\begin{equation*}
xu^2 + yu + w^2 = ( au + b)^2
\end{equation*}
for some $a,b$. Expanding this gives $x = a^2$, $y = 2ab$, $w^2 = b^2$. Thus $(x,y,w)$ lie on the singular hypersurface $V\subset \bC^3$ given by
\begin{equation}\label{e:hyper}
y^2 = 4xw^2.
\end{equation}
Let $(\fd')^{tr}$ be the image of $\fd'$ under the truncation map $\mathfrak{b} \to \mathfrak{b}^{tr}$. Then
\[
(\fd')^{tr} \cong \frac{ t\cO}{t^2 \cO} \oplus \frac{ t^2\cO}{t^4 \cO} \oplus \frac{ t^3\cO}{t^6 \cO} \oplus \frac{ t^2\cO}{t^4 \cO},
\]
which is an affine space of dimension $8$. Therefore, it follows that 
\[ 
\fd^{tr} \simeq \mathbb{A}^{5} \times V. 
\]
\end{exam}

\section{Proof of the main global results}

\subsection{Proof of Theorem \ref{t:main2}}

\begin{thm}
Under Assumptions \ref{ass:main} and \ref{ass:main2}, we have an isomorphism 
\[
\cA_{G,P} = \bigoplus_i \Gamma(X, \Omega^{d_i} ((d_i-m_i).x)).
\]
In particular, $\cA_{G,P}$ is an affine space of the same dimension as $Bun_{G,P}$.
\end{thm} 

\begin{proof}
The results of Sections \ref{s:TypeA} and \ref{s:Types} implies the $\subseteq$ inclusion of Theorem \ref{t:main}, that is:
\[
\fa_{\fg,\fp} \subseteq \bigoplus_{i=1}^n  t^{-d_i+m_i}\cO.
\]
If $(E,\phi) \in Bun_{G,P}$, then it is easy to see that the germ of $\phi$ at $x$ can be identified with an element of $\tp^\perp = t^{-1} \fn \oplus \fg (\cO)$. Therefore, the germ of $h_{G,P}(E,\phi) \in \bigoplus_{i=1}^n \Gamma( X , \Omega^{d_i}( d_i . x ) )$ at $x$ lies in $\chi( p^\perp ) \subseteq \bigoplus_{i=1}^n  t^{-d_i+m_i}\cO$. It follows that $h_{G,P}(E,\phi) \in \bigoplus_i \Gamma(X, \Omega^{d_i} ((d_i-m_i).x))$, which establishes the inclusion $\cA_{G,P} \subseteq \bigoplus_i \Gamma(X, \Omega^{d_i} ((d_i-m_i).x))$.

Since $\cA_{G,P}$ is a closed subvariety of $\bigoplus_{i=1}^n \Gamma( X , \Omega^{d_i}( d_i . x ) )$, to show that the inclusion $\cA_{G,P} \subseteq \bigoplus_i \Gamma(X, \Omega^{d_i} ((d_i-m_i).x))$ is an equality, it suffices to show both sides have the same dimension. From \cite{BKV}, we have
\[
\dim( \cA_{G,P} ) = \dim( Bun_{G,P} ) = \dim( Bun_G) + \dim(G/P) = \dim(G)(g-1) + \dim(G/P).
\]
On the other hand, by Riemann-Roch and noting that $d_i \ge m_i$, we have:
\begin{equation*}
\begin{aligned}
\dim \left( \bigoplus_i \Gamma(X, \Omega^{d_i} ((d_i-m_i).x)) \right) &= \sum_i (2d_i-1)(g-1) + (d_i-m_i) \\
&= \dim(G)(g-1) + \frac{1}{2}\left( \dim(G) - \dim(L) \right) \\
&=\dim(G)(g-1) + \dim(G/P),
\end{aligned}
\end{equation*}
where $L$ denotes the Levi subgroup of $P$. This gives the desired equality of dimensions.
\end{proof}

\begin{cor}
Let $(G,P)$ satisfy Assumption \ref{ass:main2}. Then the Hitchin map $h_{G,P} : \cM_{G,P} \to \cA_{G,P}$ is flat and surjective.
\end{cor}
\begin{proof}
In this case $\cA_{G,P}$ is non-singular, because it is an affine space. Now the result is a consequence of \cite[Corollary 11 (i)]{BKV}. 
\end{proof}

\subsection{Global image in type D}\label{ss:typeDglobal}
In this subsection, we consider the image of the parabolic Hitchin map in type $D$. In Section \ref{s:TypeD}, we defined a subspace $\fd \subset \fc(\cO) \cong \cO^n$ for which $\chi( t \tp^\perp ) \subseteq \fd$. Here we will define a global analogue of this space, or more precisely, a global analogue of $t^{-1}\bullet \fd$. Define $\cD \subseteq  \bigoplus_{j=1}^n \Gamma( X , \Omega^{d_j}( d_j x) )$ as follows. Let $(c_2 , c_4 , \dots , c_{2n-2} , p_n ) \in \bigoplus_{j=1}^{n} \Gamma( X , \Omega^{d_j}( d_j x))$. We say $(c_2 , \dots , c_{2n-2} , p_n ) \in \cD$ if 
\[
(t^2 c_2 , t^4 c_4 \dots , t^{2n-2} c_{2n-2} , t^n p_n)|_{\cO_x} \in \fd,
\]
 where $\cO_x$ is the completed local ring of $X$ at $x$, which may be identified with $\cO = \mathbb{C}[\![t]]$. Note that the resulting subvariety $\cD$ is independent of the choice of isomorphism $\cO_x \cong \cO$. We will eventually show that $\cD = \cA_{G,P}$ in type $D$.

\begin{prop}\label{p:irredcomponents}
Let $\delta_1 \ge \delta_2 \ge \dots \ge \delta_\mu$ be the sizes of the Jordan blocks of the Richardson nilpotent orbit corresponding to $\fp$. If at least one of the $\delta_j$ is odd then $\cD$ is irreducible. If all the $\delta_j$ are even, then $\cD$ consists of two irreducible components $\cD^+  , \cD^-$. In this case $\cD^+ , \cD^-$ are affine subspaces of $\bigoplus_{j=1}^n \Gamma( X , \Omega^{d_j}( d_j x))$ and are interchanged by the map $(c_2 , \dots , c_{2n-2} , p_n) \mapsto (c_2 , \dots , c_{2n-2} , -p_n)$.
\end{prop}
\begin{proof}
Let $\widetilde{\cA} = \bigoplus_{j=1}^n \Gamma( X , \Omega^{d_j}( d_j x ) )$. Define a subvariety $\cD' \subseteq \widetilde{\cA}$ as the global analogue of $t^{-1} \bullet \fd'$, that is we say $(c_2 , \dots , c_{2n-2} , p_n ) \in \cD'$ if $(t^2 c_2 , t^4c_4 , \dots , t^{2n-2}c_{2n-2} , t^n p_n)|_{\cO_x} \in \fd'$, where as above $\cO_x$ is the completed local ring of $X$ at $x$ and we choose an identification $\cO_x \cong \cO$. Then $\cD'$ is an affine subspace of $\widetilde{\cA}$ which may also be described as follows. Let $\tilde{m}_0 , \tilde{m}_1 , \dots , \tilde{m}_{2n-1}$ be the sequence given by:
\begin{equation*}
\underbrace{1,1, \dots , 1}_{\delta_1 \text{ times}}, \underbrace{2,2, \dots , 2}_{\delta_2 \text{ times}}, \dots , \underbrace{\mu,\mu, \dots , \mu}_{\delta_{\mu} -1 \text{ times}} , \mu/2.
\end{equation*}
Further, let $m_1 , m_2 , \dots , m_n$ be given by $m_j = m_{2j-1}$. Then
\begin{equation*}
\cD' = \bigoplus_{j=1}^n \Gamma( X , \Omega^{d_j}( (d_j - m_j)x ) ).
\end{equation*}
Let $(c_2 , \dots , c_{2n-2} , p_n ) \in \cD'$. For $\alpha = 0 , 1 , \dots , \mu$, let $\rho_{\alpha} \in \mathbb{C}$ be the $t^\alpha$-coefficient of $c_{\beta}$, where $\beta = 2n - \delta_1 - \delta_2 - \dots - \delta_\alpha$ and $c_{2n} = p_n^2$. In other words, $(\rho_0 , \rho_1 , \dots , \rho_\mu)$ are the coefficients of $\lambda^{2n} + c_2 \lambda^{2n-2} + \dots + c_{2n}$ lying on the boundary of the Newton polygon of slopes $-\delta_1 , -\delta_2 , \dots , -\delta_\mu$. Let $T : \cD' \to \mathbb{C}^\mu$ be the map $T(c_2 , \dots , p_n ) = ( \rho_1 , \dots , \rho_{\mu - 1} , \hat{p}_n)$, where $\hat{p}_n$ is the $t^{\mu/2}$-coefficient of $p_n$. Since $\rho_0 = 1$ and $\rho_\mu = \hat{p}_n^2$, we see that $(\rho_0 , \dots , \rho_\mu)$ can be recovered from $T(c_2 , \dots , p_n)$. Note also that $T$ is a linear map so we may choose a splitting 
\[
\cD' \cong \mathrm{Ker}(T) \oplus \mathrm{Im}(T) = \mathrm{Ker}(T) \oplus \mathbb{C}^\mu.
\] From the definition of $\cD\subseteq \cD'$, we see that $\cD \subseteq \cD'$ is given as the zero locus of a set of polynomial equations in $(\rho_1 , \dots , \rho_{\mu-1} , \hat{p}_n )$. Thus $\cD \cong \mathrm{Ker}(T) \times V$, where $V$ is a closed subvariety of $\mathbb{C}^{\mu}$. Thus to determine the irreducible components of $\cD$, we need only study the variety $V$.

From the definition of $\fd$ in Section \ref{s:TypeD}, we see that $V$ is defined by requiring that for each boundary segment of the Newton polygon with even slope, an associated polynomial is a square. This may be expressed in the following way. Let $I \subseteq \{ 0 ,1 , \dots , \mu \}$ be the set of integers $0 \le j \le \mu$ such that the corresponding coefficient $\rho_j$ lies on a segment of even slope. Note that $\rho_j$ corresponds to the vertex joining line segments of slopes $-\delta_{j-1}$ and $-\delta_j$. In other words, $j \in I$ if $\delta_{j-1}$ or $\delta_j$ is even. Define a {\em segment} of $I$ to be a consecutive sequence of integers $a , a+1 , \dots , b \in I$, such that $a-1 , b+1 \notin I$. Denote such a segment by $[a , b] \subseteq I$. Then $V$ is defined as follows. For each segment $[a , b] \subseteq I$, let $q_{[a,b]}(u) \in \mathbb{C}[u]$ be given by:
\begin{equation*}
q_{[a,b]}(u) = \rho_a u^{b-a} + \rho_{a+1} u^{b-a-1} + \dots + \rho_b.
\end{equation*}
Note that if $[a,b]$ is a segment, then $b > a$ and $b-a$ is even. The subvariety $V \subseteq \mathbb{C}^{\mu}$ is defined by requiring that each polynomial $q_{[a,b]}$ is a square. Suppose that $[a_1 , b_1] , [a_2 , b_2] , \dots , [a_r , b_r]$ are the distinct segments of $I$, ordered so that $a_1 < b_1 < a_2 < \dots < b_{r-1} < a_r < b_r$. It follows that $V$ can be decomposed into a product $V \cong V_1 \times \dots \times V_r \times \mathbb{C}^d$ for some $d$. More precisely, we define $V_j$ in the following way\footnote{Note that the Hitchin base is singular whenever there is an occurrence of case (1) or case (3).}:
\begin{enumerate}
\item{If $a_j \neq 0$ and $b_j \neq \mu$ then $V_j \subseteq \mathbb{C}^{b_j - a_j}$ is the set of $(\rho_{a_j} , \dots , \rho_{b_j})$ such that the corresponding polynomial $q_{[a_j , b_j]}(u)$ is a square. Such a $V_j$ is irreducible.}
\item{If $a_j = 0$ and $b_j \neq \mu$ then $j = 1$ and $V_1 \subseteq \mathbb{C}^{b_1 - 1}$ is the set of $(\rho_1 , \dots , \rho_{b_1} )$ such that $q_{[0 , b_1]}(u) = u^{b_1} + \rho_1 u^{b_1 - 1} + \dots + \rho_{b_1}$ is a square. Thus $V_1$ is irreducible, in fact isomorphic to an affine space.}
\item{If $a_j \neq 0$ and $b_j = \mu$ then $j = r$ and $V_r \subseteq \mathbb{C}^{\mu - a_r}$ is the set of $(\rho_{a_r} , \dots , \rho_{\mu - 1} , \hat{p}_n )$ such that $q_{[a_r , \mu]}(u) = \rho_{a_r} u^{\mu-a_r} + \dots + \rho_{\mu-1} u + \hat{p}_n^2$ is a square. Such a $V_{r}$ is irreducible.}
\item{If $a_j = 0$ and $b_j = \mu$, then $j=1$ and $I = [0 , \mu]$. This case occurs if and only if the $\delta_i$ are all even. In this case $V_1$ is the variety of points $(\rho_1 , \dots , \rho_{\mu-1} , \hat{p}_n )$ such that $u^{\mu} + \rho_1 u^{\mu-1} + \dots + \rho_{\mu-1} u + \hat{p}_n^2$ is a square. It is easy to see that $\rho_1 , \dots , \rho_{\mu/2}$ may be chosen arbitrarily and that $\rho_{\mu/2 + 1} , \dots , \rho_{\mu-1} , \hat{p}_n^2$ may be expressed as certain polynomials in $\rho_1 , \rho_2 , \dots , \rho_{\mu/2}$. In particular, we have an equation of the form
\begin{equation*}
\hat{p}_n^2 = f( \rho_1 , \rho_2 , \dots , \rho_{\mu/2} ),
\end{equation*}
for some polynomial $f \in \mathbb{C}[ \rho_1 , \rho_2 , \dots , \rho_{\mu/2}]$. We claim that $f = g^2$ for some $g \in \mathbb{C}[\rho_1 , \rho_2 , \dots , \rho_{\mu/2}]$. To see this, consider the condition for $u^{\mu} + \rho_1 u^{\mu-1} + \dots + \hat{p}_n^2$ to be a square:
\begin{equation*}
u^{\mu} + \rho_1 u^{\mu-1} + \dots + \hat{p}_n^2 = ( u^{\mu/2} + \sigma_1 u^{\mu/2 - 1} + \dots + \sigma_n)^2
\end{equation*}
for some $\sigma_1 , \dots , \sigma_n$. Expanding and equating coefficients, we see that $\sigma_1 , \sigma_2 , \dots , \sigma_n$ can be expressed as polynomials in $\rho_1 , \rho_2 , \dots , \rho_{\mu/2}$. In particular, $\sigma_n = g(\rho_1 , \dots , \rho_{\mu/2})$ for some $g \in \mathbb{C}[\rho_1 , \dots , \rho_{\mu/2} ]$. Equating $u^0$ coefficients now gives:
\begin{equation*}
\hat{p}_n^2 = \sigma_n^2 = g^2(\rho_1 , \rho_2 , \dots , \rho_{\mu/2} ).
\end{equation*}
Thus either $\hat{p}_n = g(\rho_1 , \dots , \rho_{\mu/2})$ or $\hat{p}_n = -g(\rho_1 , \dots , \rho_{\mu/2})$. This shows that $V_1$ has two irreducible components $V_1^+ , V_1^-$ which are interchanged under the map $(\rho_1 , \dots , \rho_{\mu-1} , \hat{p}_n) \mapsto (\rho_1 , \dots , \rho_{\mu - 1} , -\hat{p}_n )$. Moreover $V_1^+$ and $V_1^-$ are isomorphic to affine spaces, since on each of these components we can write $\rho_{\mu/2+1} , \dots , \rho_{\mu - 1} , \hat{p}_n$ as polynomials in $\rho_1 , \dots , \rho_{\mu/2}$.}
\end{enumerate}
Taking into consideration the above cases, we see that $V \cong V_1 \times V_1 \times \dots \times V_r \times \mathbb{C}^d$ and hence $\cD$ is irreducible, except when all the $\delta_j$ are all even, in which case there are two components which are exchanged by the map $(c_2 , \dots , c_{2n-2} , p_n) \mapsto (c_2 , \dots , c_{2n-2} , -p_n)$.
\end{proof}

\begin{thm}\label{t:basetypeD}
Suppose that not all $\delta_j$ are even. Then we have an equality
\begin{equation*}
\cA_{G,P} = \cD.
\end{equation*}
In the case that all the $\delta_j$ are even, then $\cA_{G,P}$ is one of the two irreducible components $\cD^+,\cD^-$ of $\cD$ (the two components are exchanged by the outer automorphism of $\mathfrak{so}_{2n}$).
\end{thm}

\begin{rem}\label{r:plusminus}
In the case that all the $\delta_j$ are even, there are two nilpotent orbits of $\mathfrak{so}_{2n}$ with Jordan blocks of sizes $\delta_1 , \dots , \delta_\mu$. These are the Richardson orbits for two parabolic subalgebras $\fp_+ , \fp_-$ and the two components of $\cD$ give the Hitchin bases for $\fp_+  \fp_-$. It is possible to determine precisely which parabolic corresponds to which component of $\cD$ using the results of \cite{Spaltenstein3}, however the details are quite complicated.
\end{rem}

\begin{proof}
By Proposition \ref{p:localbaseD}, we have inclusions $\cA_{G,P} \subseteq \cD \subseteq \widetilde{\cA}$. From \cite{BKV}, we have that $\cA_{G,P}$ is a closed irreducible subvariety of $\widetilde{\cA}$ of dimension $\dim( \Bun_{G,P} ) = (g-1)\dim(G) + \dim(G/P)$. By Proposition \ref{p:irredcomponents}, it suffices to show that $\cA_{G,P}$ is an irreducible component of $\cD$. Thus it suffices to show that $\dim \left( \cD  \right) = \dim( \Bun_{G,P}) = \dim( \Bun_G ) + \dim( \mathfrak{n})$. Let $\mathcal{B} = \bigoplus_{j=1}^n \Gamma( X , \Omega^{d_j} )$. Then $\dim( \mathcal{B}) = \dim(\Bun_G)$ and $\mathcal{B} \subseteq \cD$. Thus we only need to show that the codimension of $\mathcal{B}$ in $\cD$ is $\dim(\mathfrak{n})$. Define $\fd^{tr}$ and $\mathfrak{b}^{tr}$ as in the proof of Proposition \ref{p:localbaseD}. Then $\fd^{tr} \subseteq \mathfrak{b}^{tr}$. To show that the codimension of $\mathcal{B} \subseteq \cD$ is $\dim(\mathfrak{n})$ is equivalent to showing that the codimension of $\fd^{tr} \subseteq \mathfrak{b}^{tr}$ is
\begin{equation*}
\sum_{j=1}^n d_j - \dim(\mathfrak{n}) = (2 + 4 + \dots + 2n-2) + n - \dim(\mathfrak{n}) = n^2 - \dim(\mathfrak{n}).
\end{equation*}
Let $\delta_1 \ge \delta_2 \ge \dots \ge \delta_{\mu}$ be the Jordan blocks of the Richardson nilpotent orbit associated to $\fp$. Recall that for each $j$, $\delta_{2j-1}$ and $\delta_{2j}$ are either both even or both odd. Let $n^{ev}$ be the number of such pairs for which $\delta_{2j-1},\delta_{2j}$ are even and $n^{odd}$ the number of such pairs where $\delta_{2j-1},\delta_{2j}$ are odd. So $\mu/2 = n^{ev} + n^{odd}$. Let $(\fd')^{tr}$ be the subset of $\mathfrak{b}^{tr}$ consisting of $(c_2 , c_4 , \dots , c_{2n-2} , p_n ) \in \mathfrak{b}^{tr}$ which lie in the Newton polygon with slopes $-\delta_1 , -\delta_2 , \dots , -\delta_{\mu}$ (in other words, $(\fd')^{tr}$ is the image of $\fd'$ under the projection $\pi : \mathfrak{b} \to \mathfrak{b}^{tr}$ defined in the proof of Proposition \ref{p:localbaseD}). We have a sequence of inclusions
\begin{equation*}
\fd^{tr} \subseteq (\fd')^{tr} \subseteq \mathfrak{b}^{tr}.
\end{equation*}
From Section \ref{s:TypeD} we see that $\fd^{tr} \subseteq (\fd')^{tr}$ is defined by the vanishing of $n^{ev}$ independent polynomial equations (independent in the sense that their differentials are generically linearly independent). Thus $\fd^{tr}$ is an affine subvariety of $(\fd')^{tr}$ of codimension $n^{ev}$. So we are reduced to showing that the codimension of $(\fd')^{tr}$ in $\mathfrak{b}^{tr}$ is $n^2 - \dim(\mathfrak{n}) - n^{ev}$. Let $m_1 , m_2 , \dots , m_n$ be defined as in the proof of Proposition \ref{p:irredcomponents}. Then by the definition of $(\fd')^{tr}$, one sees that its codimension in $\mathfrak{b}^{tr}$ is exactly
\begin{equation*}
\sum_{j=1}^n m_j = \frac{1}{2} \sum_{j=1}^{\mu} j \delta_j  + \frac{1}{2} n^{odd}  -\frac{\mu}{2}.
\end{equation*}
Let $n'_1 \ge n'_2 \ge \dots \ge n'_k$ be the conjugate partition of $\delta_1 \ge \dots \ge \delta_\mu$ and recall the following identity relating a partition of $2n$ and its conjugate:
\begin{equation}\label{e:dim1}
-2n + 2\sum_{j=1}^{\mu} j \delta_j = (n'_1)^2 + (n'_2)^2 + \dots + (n'_k)^2.
\end{equation}
Recall also from \cite{KP} that the dimension of the Levi $\fl \subset \fp$ in type $D_{2n}$ is given by:
\begin{equation}\label{e:dim2}
\dim( \fl ) = \frac{1}{2} \sum_j (n'_j)^2 - n^{odd}.
\end{equation}
Combining Equations (\ref{e:dim1}) and (\ref{e:dim2}), we see that:
\begin{equation*}
\begin{aligned}
\frac{1}{2} \sum_{j=0}^{\mu} j \delta_j  + \frac{1}{2} n^{odd}  -\frac{\mu}{2} &= \frac{ \dim( \fl ) + n^{odd} }{2} + \frac{n}{2} + \frac{n^{odd}}{2} - \frac{\mu}{2} \\
&= \frac{ \dim( \fl ) + n }{2} + n^{odd} - \frac{\mu}{2} \\
&= \frac{ \dim( \mathfrak{so}_{2n}) - 2\dim( \mathfrak{n} ) + n }{2} + n^{ev} \\
&= n^2 - \dim( \mathfrak{n} ) + n^{ev}.
\end{aligned}
\end{equation*}
This shows that $\cD$ has the correct dimension and hence $\cA_{G,P}$ is an irreducible component of $\cD$.
\end{proof}

\section{Completion of the proof of the main local results}\label{s:completion}

Recall that $\fa_{\fg,\fp} \subseteq \bigoplus_{j=1}^n t^{-d_j} \cO$. We define a space $\hat{\fa}_{\fg,\fp} \subseteq \bigoplus_{j=1}^n t^{-d_j} \cO$ as follows. If Assumption \ref{ass:main2} holds, then we set
\[
\hat{\fa}_{\fg,\fp} = \bigoplus_{j=1}^n t^{-d_j+m_j} \cO.
\]
Suppose we are in type $D$. Let $\delta_1 , \dots , \delta_\mu$ be the sizes of the Jordan blocks of the Richardson nilpotent orbit corresponding to $\fp$. If the $\delta_j$ are not all even, we let $\hat{\fa}_{\fg,\fp} = t^{-1} \bullet \fd$, where $\fd$ is defined as in Section \ref{s:TypeD}. In the special case where all the $\delta_j$ are even, recall from Theorem \ref{t:basetypeD} that $\cD = \cD^+ \cup \cD-$ is a union of two irreducible components. There is a corresponding local version of this statement, namely $t^{-1} \bullet \fd = t^{-1} \bullet \fd^+ \cup t^{-1} \bullet \fd^-$. It is not hard to see that we have either $\chi(\tp^\perp) \subseteq t^{-1} \bullet \fd^+$ or $\chi(\tp^\perp) \subseteq t^{-1} \bullet \fd^-$ depending on whether $\fp = \fp_+$ or $\fp_-$, where $\fp_\pm$ are the two parabolics whose corresponding Richardson orbit has Jordan blocks $\delta_1 , \dots , \delta_\mu$ (cf. Remark \ref{r:plusminus}). We then let $\hat{\fa}_{\fg,\fp}$ be whichever of $t^{-1} \bullet \fd^\pm$ satisfies $\chi(\tp^\perp) \subseteq t^{-1} \bullet \fd^\pm$. In the case that we are in type $D$ and Assumption \ref{ass:main2} holds, then one can easily check that both these definitions give the same space.
\begin{thm}
We have an equality $\fa_{\fg,\fp} = \hat{\fa}_{\fg,\fp}$. In particular, if Assumption \ref{ass:main2} holds, then 
\[
\fa_{\fg,\fp} = \bigoplus_{j=1}^n t^{-d_j+m_j} \cO.
\]
\end{thm}
The rest of this section is about proving this theorem. 
From Sections \ref{s:TypeA}, \ref{s:Types} and Proposition \ref{p:localbaseD}, we have shown an inclusion
\[
\chi( \tp^\perp) \subseteq \hat{\fa}_{\fg,\fp}.
\]
We argue that $\hat{\fa}_{\fg,\fp}$ is the closure of $\chi( \tp^\perp )$ in the following way. We fix a positive integer $N$ and work modulo $t^N$. More precisely, define $\fb(N)$ to be
\[
\fb(N) = \bigoplus_{j=1}^n \left( t^{-d_j} \cO / t^N \cO \right).
\]
There is a natural surjection $\pi : \bigoplus_{j=1}^n t^{-d_j} \cO \to \fb(N)$ and we may consider the image $\hat{\fa}_{\fg,\fp}(N)$ of $\hat{\fa}_{\fg,\fp}$ under $\pi$, in other words, we consider $\hat{\fa}_{\fg,\fp}$ modulo $t^N$. Note that $\fb(N)$ and $\hat{\fa}_{\fg,\fp}(N)$ are finite dimensional affine varieties. The composition $\pi \circ \chi$ gives a map
\[
\pi \circ \chi : \tp^\perp \to \hat{\fa}_{\fg,\fp}(N).
\]
In view of Remark \ref{r:topology}, we are reduced to the following:

\begin{lem}\label{l:dense}
For each positive integer $N$, the image of the composition
\[
\pi \circ \chi : \tp^\perp \to \hat{\fa}_{\fg,\fp}(N)
\]
is dense in $\hat{\fa}_{\fg,\fp}(N)$.
\end{lem}
\begin{proof}
Our proof uses global methods. Let $X$ be a smooth complex projective curve of genus $g > 1$ and let $x \in X$ be a marked point. Fix also an isomorphism of local rings $\cO_x \cong \cO$, where $\cO_x$ is the completed local ring of $X$ at $x$. Let $m$ be a positive integer and $y_1, y_2 , \dots , y_m$ be $m$ distinct points of $X - x$. Let $B$ be a Borel subgroup of $G$. We consider the stack 
\[
\cY = \cM_{G , (P,x) , (B,y_1) , \dots , (B ,y_m)}
\]
of parabolic $G$-Higgs bundles on $X$ with reduction to $P$ at $x$ and reductions to $B$ at the remaining marked points $y_1 , \dots , y_m$. The Hitchin map for $\cY$ has the form
\[
h : \cY \to \cB = \bigoplus_{j=1}^n \Gamma( X, \Omega^{d_j}( d_j(x+y_1 + \dots + y_m))).
\]
Let $\cA \subset \cB$ denote the Zariski closure of the image of $h$. We can describe $\cA$ by relating it to its local counterpart as follows. 
Let 
\[
\theta : \cB \to \fb(N)
\]
be the natural localisation map. In other words, this is the map
\[
\theta : \bigoplus_{j=1}^n \Gamma( X, \Omega^{d_j}( d_j(x+y_1 + \dots + y_m))) \to \bigoplus_{j=1}^n \left( t^{-d_j} \cO / t^N \cO \right)
\]
which for each $j$, takes a section of $\Omega^{d_j}( d_j(x+y_1 + \dots + y_m))$ and restricts it to its germ at $x$ in $t^{-d_j} \cO$, taken modulo $t^N$. Then from the definition of $\hat{\fa}_{\fg,\fp}$ and from Theorems \ref{t:main2} and \ref{t:basetypeD}, it follows that 
\[
\cA = \{ \alpha \in \cB \; | \; \theta(\alpha) \in \hat{\fa}_{\fg,\fp}(N) \}.
\]
Thus, we have a Cartesian square
\begin{equation*}\xymatrix{
\cA \ar[r] \ar[d]^-{\theta|_{\cA}} & \cB \ar[d]^-\theta \\
\hat{\fa}_{\fg,\fp}(N) \ar[r] & \fb(N)
}
\end{equation*}
where the horizontal arrows are inclusions. The kernel of $\theta : \cB \to \fb(N)$ is exactly the sections $(b_1 , \dots , b_n) \in \bigoplus_{j=1}^n \Gamma( X , \Omega^{d_j}( d_j(x+y_1+\dots + y_m)))$ such that for $1 \le j \le n$, we have that $b_j$ vanishes to order $N$ at $x$. That is,
\[
\mathrm{Ker}( \theta ) = \bigoplus_{j=1}^n \Gamma( X , \Omega^{d_j}( -Nx + d_j(y_1 + \dots + y_m) ) ).
\]
Now we choose $m$ such that $m > N$. In this case $\deg( \Omega^{d_j}( -Nx + d_j(y_1 + \dots + y_m) ) ) > 2g-2$, hence by Riemann-Roch, we find
\begin{equation*}
\begin{aligned}
\dim( \cB ) - \dim( \mathrm{Ker}(\theta)) &= \dim\left( \bigoplus_{j=1}^n \Gamma( X , \Omega^{d_j}( d_j(x+y_1+\dots + y_m))) \right) \\
& \; \; \; \; - \dim \left( \bigoplus_{j=1}^n \Gamma( X , \Omega^{d_j}( -Nx + d_j(y_1 + \dots + y_m) ) ) \right) \\
&= nN + \sum_{j=1}^n d_j \\
&= \dim( \fb(N) ).
\end{aligned}
\end{equation*}

In particular, we see that $\theta$ is surjective. Since the diagram is Cartesian, it follows that $\theta|_\cA$ is also surjective. Now by our main global results, $h(\cY)$ is dense in $\cA$; thus, $\theta(h(\mathcal{Y}))$ is dense in $\theta(\cA)=\hat{\fa}_{\fg,\fp}(N)$. On the other hand, by definition, we have the inclusions 
\[
 \theta(h(\cY))\subset (\pi\circ \chi)(\tp^\perp)\subset \hat{\fa}_{\fg,\fp}(N).
\]
 Thus $(\pi\circ \chi)(\tp^\perp)$ is also dense in $\hat{\fa}_{\fg,\fp}(N)$.
\end{proof}


  \end{document}